\documentclass[12pt]{amsart}
\usepackage{amssymb, amsmath, amsthm, latexsym, array,euscript}
\usepackage{fullpage}
\usepackage{verbatim}
\usepackage{tikz}
\usetikzlibrary{arrows,shapes,trees}
\usepackage[scr=boondoxo,scrscaled=1.01]{mathalfa}

\usepackage{xcolor}
\definecolor{darkblue}{rgb}{0,0,.5}
\definecolor{darkgreen}{rgb}{.2,0.5,.2}
\usepackage[colorlinks=true,linkcolor=darkblue,urlcolor=violet,citecolor=magenta]{hyperref}

\numberwithin{equation}{section}

\input{cyracc.def}

\font\tencyr=wncyr10 
\font\tencyi=wncyi10 
\font\tencysc=wncysc10 

\def\rus{\tencyr\cyracc}
\def\rusi{\tencyi\cyracc}
\def\rusc{\tencysc\cyracc}

\newtheorem{thm}{Theorem}[section]
\newtheorem{conj}[thm]{Conjecture}
\newtheorem{lm}[thm]{Lemma}
\newtheorem{cl}[thm]{Corollary}
\newtheorem{prop}[thm]{Proposition}

\theoremstyle{remark}
\newtheorem{ex}[thm]{Example}
\newtheorem{rmk}[thm]{Remark}

\theoremstyle{definition}
\newtheorem{df}{Definition}

\newcommand {\be}{{\mathfrak b}}
\newcommand {\ce}{{\mathfrak c}}

\newcommand {\g}{{\mathfrak g}}
\newcommand {\h}{{\mathfrak h}}

\newcommand {\el}{{\mathfrak l}}

\newcommand {\n}{{\mathfrak n}}
\newcommand {\N}{{\mathfrak N}}
\newcommand {\p}{{\mathfrak p}}
\newcommand {\q}{{\mathfrak q}}
\newcommand {\rr}{{\mathfrak r}}
\newcommand {\es}{{\mathfrak s}}
\newcommand {\te}{{\mathfrak t}}
\newcommand {\ut}{{\mathfrak u}}

\newcommand {\sln}{{\mathfrak {sl}}_n}

\newcommand {\spn}{{\mathfrak {sp}}_{2n}}

\newcommand {\sone}{{\mathfrak {so}}_{2n}}

\newcommand {\soN}{{\mathfrak {so}}_{N}}
\newcommand {\soVB}{{\mathfrak {so}}(\VV, \caB)}

\newcommand {\eus}{\EuScript}

\newcommand {\gA}{{\eus A}}

\newcommand {\gN}{{\eus N}}

\newcommand {\gS}{{\eus S}}
\newcommand {\gZ}{{\eus Z}}

\newcommand {\caB}{{\mathscr{B}}}
\newcommand {\caL}{{\mathscr{L}}}

\newcommand {\esi}{\varepsilon}
\newcommand {\ap}{\alpha}

\newcommand {\lb}{\lambda}
\newcommand {\vp}{\varphi}
\newcommand {\vth}{\vartheta}

\newcommand {\bia}{{\un{\boldsymbol{i}}}}
\newcommand {\bja}{{\un{\boldsymbol{j}}}}

\newcommand {\ov}{\overline}
\newcommand {\un}{\underline}

\newcommand{\gt}{\mathfrak}

\newcommand{\SL}{{\rm SL}}
\newcommand{\GL}{{\rm GL}}

\newcommand{\Aut}{\mathsf{Aut}}

\newcommand{\ad}{\mathrm{ad}}

\newcommand {\hot}{{\mathsf{ht}}}

\newcommand{\ind}{{\rm ind\,}}

\newcommand{\rk}{\mathrm{rk\,}}

\newcommand{\Lie}{\mathsf{Lie}}

\newcommand{\Op}{\mathrm{O}}

\newcommand {\co}{{\mathcal O}}

\newcommand {\cz}{{\mathcal Z}}

\newcommand {\trdeg}{{\mathrm{tr.deg\,}}}
\newcommand {\ggs}{{\sf g.g.s.{}}}

\newcommand {\bbk}{{\Bbbk}}

\newcommand {\VV}{{\mathsf V}}

\newcommand {\md}{/\!\!/}

\newcommand {\GR}[2]{{\textrm{{\sf\bfseries #1}}}_{#2}}
\newcommand {\GRt}[2]{ {\widetilde{\textrm{\sf\bfseries #1}}}_{#2} }

\newcommand {\BZ}{{\mathbb Z}}
\newcommand {\BN}{{\mathbb N}}
\newcommand {\beq}{\begin{equation}}
\newcommand {\eeq}{\end{equation}}

\newcommand {\rg}{{\rangle}}
\renewcommand {\lg}{{\langle}}

\renewcommand{\le}{\leqslant}
\renewcommand{\ge}{\geqslant}

\newcommand {\BP}{{\mathbb P}}

\newcommand{\bb}{\boldsymbol{b}}

\newcommand{\Ord}{\mathsf{ord}(\sigma)}

\begin{document}
\hfill {\scriptsize November 19, 2022}
\vskip1ex

\title[Some remarks on periodic contractions]{Automorphisms of finite order, periodic 
contractions, and Poisson-commutative subalgebras of $\gS(\g)$}
\author[D.\,Panyushev]{Dmitri I. Panyushev}
\address[D.P.]%
{Institute for Information Transmission Problems of the R.A.S.,  
Moscow 127051, Russia}
\email{panyushev@iitp.ru}
\author[O.\,Yakimova]{Oksana S.~Yakimova}
\address[O.Y.]{Institut f\"ur Mathematik, Friedrich-Schiller-Universit\"at Jena,  07737 Jena,
Deutschland}
\email{oksana.yakimova@uni-jena.de}
\thanks{The research of the first author is supported by the R.F.B.R. grant {\rus N0} 20-01-00515.
The second author is funded by the DFG (German Research Foundation) --- project number 404144169.}
\dedicatory{To Victor Kac with admiration  }
\begin{abstract}
Let $\g$ be a semisimple Lie algebra, $\vth\in {\sf Aut}(\g)$ a finite order automorphism, and $\g_0$ the subalgebra of fixed points of $\vth$. Recently, we noticed that using $\vth$ one can construct a pencil of
compatible Poisson brackets on $\gS(\g)$, and thereby a `large' Poisson-commutative subalgebra 
$\gZ(\g,\vth)$ of $\gS(\g)^{\g_0}$. In this article, we study invariant-theoretic properties of $(\g,\vth)$
that ensure good properties of $\gZ(\g,\vth)$. Associated with $\vth$ one has a natural Lie algebra 
contraction $\g_{(0)}$ of $\g$ and the notion of a {\it good generating system} (=g.g.s.) in 
$\gS(\g)^\g$. We prove that in many cases the equality $\ind\g_{(0)}=\ind\g$ holds
and $\gS(\g)^\g$ has a {g.g.s.} According to V.G.\,Kac's classification of finite order automorphisms (1969), 
$\vth$ can be represented by a Kac diagram, $\eus K(\vth)$, and our results often use this presentation. The most surprising observation is that $\g_{(0)}$ depends only on the set of nodes in $\eus K(\vth)$ with nonzero labels, and that if $\vth$ is inner and a certain label is nonzero, then $\g_{(0)}$ is isomorphic to 
a parabolic contraction of $\g$.  
\end{abstract}
\maketitle

\tableofcontents
\section{Introduction}   \label{sect:intro}

\subsection{}
Completely integrable Hamiltonian systems on symplectic algebraic varieties are fundamental objects having a 
rich structure. They have been extensively studied  from different points of view in various areas of mathematics such as 
differential geometry, classical mechanics, algebraic and Poisson geometries, and more recently, representation theory. 
A natural choice for the underlying variety is a coadjoint orbit of an algebraic Lie algebra $\q$.  
In this context, one may obtain an integrable system from a {\it Poisson 
commutative} ({=}{\sf PC}) subalgebra of the symmetric 
algebra $\gS(\q)$. 
As is well-known,  $\gS(\q)$ has the standard Lie--Poisson structure $\{\ ,\,\}$.

In this paper, the base field $\bbk$ is algebraically closed, ${\sf char}\,\bbk=0$, and 
$\g$ is the  Lie algebra 
of a connected reductive algebraic group $G$. 
Let $\eus U(\g)$ be the enveloping algebra of $\g$. 
We are interested in {\sf PC} subalgebras of $\gS(\g)^\h$, where $\h=\Lie(H)$ and 
$H\subset G$ is a connected reductive subgroup. 
These subalgebras are closely related to 
commutative subalgebras of ${\eus U}(\g)^{\h}$ and thereby to branching rules involving $G$ and $H$, see \cite[Sect.\,6.1]{OY} for some examples. Note also that the centre of ${\eus U}(\g)^{\h}$ is described in~\cite[Theorem\,10.1]{knop}.

Whenever a PC subalgebra of $\gS(\g)^\h$ is large enough, one extends it to a PC subalgebra of $\gS(\g)$, which provides 
completely integrable systems on generic orbits. This idea is employed in 
\cite{gs1,gs2}, where the foundation of a beautiful geometric theory has also been laid.  

The Lenard--Magri scheme provides a method for constructing ``large'' {\sf PC} subalgebras via 
compatible Poisson brackets. Let $\{\ ,\,\}'$ be another Poisson bracket on $\gS(\g)$ compatible with 
$\{\ ,\,\}$ and $\{\ ,\,\}_t= \{\ ,\,\}+ t\{\ ,\,\}'$. Using the centres of the Poisson algebras $(\gS(\g), \{\ ,\,\}_t)$ 
for regular values of  $t$, one obtains a {\sf PC} subalgebra $\gZ\subset \gS(\g)$,  
see Section~\ref{subs:compatible} for details. Here the main 
questions are:

\textbullet \quad  how to find/construct an appropriate compatible bracket $\{\ ,\,\}'$ ?

\textbullet \quad  what are the properties of {\sf PC} subalgebras $\gZ$ obtained?

\textbullet \quad  is it possible to quantise $\gZ$, i.e., lift it to $\eus U(\g)$ ?

\noindent
A well-known approach that exploits a Poisson bracket with a ``frozen'' argument as $\{\ ,\,\}'$ provides the Mishchenko--Fomenko subalgebras of $\gS(\g)$~\cite{bols}, and their quantisation is studied
in~\cite{rybn,fft,m-y,4-er}.

In recent articles~\cite{OY,bn-oy,py21}, we develop new methods for constructing $\{\ ,\,\}'$ and for studying the corresponding {\sf PC} subalgebras $\gZ$. 

{\bf (A)} In~\cite{OY}, we prove that any involution of $\g$ yields a compatible Poisson bracket on 
$\gS(\g)$ and consider the related {\sf PC} subalgebras of $\gS(\g)$. A generalisation of this approach to 
$\vth\in {\sf Aut}(\g)$ of arbitrary finite order is presented in~\cite{py21}. The latter heavily relies on 
Invariant Theory of $\vth$-groups developed by E.B.\,Vinberg in~\cite{vi76}.

{\bf (B)}  In~\cite{bn-oy}, we study compatible Poisson brackets related to a vector space sum 
$\g=\rr\oplus\h$, where $\rr,\h$ are subalgebras of $\g$. To expect some good properties of $\gZ$, one 
has to assume here that at least one of the subalgebras is spherical in $\g$.

In both cases, we get two compatible linear Poisson brackets $\{\ ,\,\}'$ and $\{\ ,\,\}''$ such that $\{\ ,\,\}=\{\ ,\,\}'+\{\ ,\,\}''$ is the initial Lie--Poisson structure and study the pencil of Poisson brackets
\[
    \text{$\{\ ,\,\}_{t}=\{\ ,\,\}'+t\{\ ,\,\}''$, \quad $t\in \BP^1=\bbk\cup\{\infty\}$}, 
\]
where $\{\ ,\,\}_{\infty}=\{\ ,\,\}''$. Each bracket $\{\ ,\,\}_{t}$ provides a Lie algebra structure on the 
vector space $\g$, denoted by $\g_{(t)}$. The brackets with $t\in\bbk^*{:}{=}\bbk\setminus\{0\}$ 
comprise Lie algebras isomorphic to $\g=\g_{(1)}$, while the Lie algebras $\g_{(0)}$ and $\g_{(\infty)}$ 
are different. Since both are contractions of the initial Lie algebra $\g$, we have 
$\ind\g_{(0)}\ge \ind\g$ and $\ind\g_{(\infty)}\ge \ind\g$.
\\ \indent
In case {\bf (A)}, the role of the Lie algebras $\g_{(0)}$ and $\g_{(\infty)}$ is not symmetric. The algebra 
$\g_{(\infty)}$ is nilpotent, while a maximal reductive subalgebra of $\g_{(0)}$ is $\g^\vth$. Roughly 
speaking, the output of \cite{OY,py21} is that in order to expect some good properties of the {\sf PC}
subalgebra $\gZ=\gZ(\g,\vth)$, one needs (at least) the following two properties of $\vth$:
\begin{itemize}
\item[\sf (i)] \ $\ind\g_{(0)}= \ind\g$;
\item[\sf (ii)] \ the algebra $\gS(\g)^\g$ contains a {\it good generating system} ({\sf g.g.s.}) with
respect to $\vth$, see Section~\ref{subs:theta} for details. (Then we also say that $\vth$ {\it admits\/} 
a {\sf g.g.s.})
\end{itemize}
The Lie algebra $\g_{(0)}$ is said to be the $\vth$-{\it contraction\/} or a {\it periodic contraction\/} of $\g$.

\subsection{} This article is a sequel to~\cite{py21}. It is devoted to invariant-theoretic properties of 
a $\BZ_m$-graded simple Lie algebra $\g$, which is motivated by our study of {\sf PC} subalgebras of
$\gS(\g)$. We concentrate on proving 
{\sf (i)} and {\sf (ii)} for various types of $\g$ and $\vth\in{\sf Aut}(\g)$. Accordingly, we establish some
good properties of related {\sf PC} subalgebras. Let ${\sf Aut}^f(\g)$ (resp. ${\sf Int}^f(\g)$) be the set of 
all (resp. inner) automorphisms of $\g$ of finite order. For $\vth\in {\sf Aut}^f(\g)$, we also say that $\vth$ 
is {\it periodic}. Let $m=|\vth|$ be the order of $\vth$ and $\zeta=\sqrt[m]1$ a fixed primitive root of unity. If 
$\g_i$ is the eigenspace of $\vth$ corresponding to $\zeta^i$, then $\g=\bigoplus_{i=0}^{m-1}\g_i$ is the 
$\BZ_m$-grading of $\g$ associated with $\vth$. A classification of periodic automorphisms of $\g$ is due 
to V.\,Kac~\cite{k69}, and our results often invoke the {\it Kac diagram} of $\vth$. We refer 
to~\cite[\S\,8]{vi76}, \cite[Chap.\,3,\,\S\,3]{t41} and \cite[Ch.\,8]{kac}
for generalities on Kac's classification and the Kac 
diagrams. The Kac diagram of $\vth$, $\eus K(\vth)$, is an affine Dynkin diagram of $\g$ (twisted, if 
$\vth$ is outer) endowed with nonnegative integral labels. We recall the relevant setup and give an explicit construction of $\vth$ via 
$\eus K(\vth)$, see Sections~\ref{subs:kac}, \ref{sect:inner-auto}, and~\ref{sect:outer}.

Actually, Kac's classification stems from the study of $\BZ$-gradings of ``his'' infinite-dimensional Lie 
algebras \cite{k69}. Our recent results on $\g_{(0)}$ and $\gZ(\g,\vth)$ have applications to the infinite-dimensional 
case, too~\cite[Sect.\,8]{py21}. However, in this article, we do not refer explicitly to Kac--Moody algebras, 
which agrees with the  approach taken in \cite{t41}.

It is known that $\ind\g_{(0)}= \ind\g$, if $m=2$~\cite{p07} or $\g_1$ contains regular elements of 
$\g$~\cite{p09}. Here we prove equality {\sf (i)} for $\ind\g_{(0)}$ in the following cases:
\begin{enumerate}
\item either $m=3$ or $m=4,5$ and the $G_0$-action on $\g_1$ is stable, see 
Section~\ref{sect:index-g_0}; 
\item $\vth$ is inner and a certain label on the Kac diagram of $\vth$ is nonzero, see Theorem~\ref{thm:semidir} and Proposition~\ref{prop:parab&index};
\item $\vth$ is an arbitrary {\bf inner} automorphism of $\g=\sln$, see Proposition~\ref{prop:An}; 
\item $\vth\in {\sf Aut}^f(\spn)$ and $m$ is odd, see Proposition~\ref{prop:ind-sp-odd};
\item $\vth$ is an arbitrary automorphism of $\GR{G}{2}$ (Example~\ref{ex:G2}) or of $\gt{so}_N$, see Section~\ref{sect:son}.
\end{enumerate}
Our proofs for {\it\bfseries (3)}-{\it\bfseries (5)} rely on a new result that $\g_{(0)}$ depends only 
on the set of nodes in $\eus K(\vth)$ with nonzero labels, i.e., having replaced all nonzero labels with 
`1', one obtains the same periodic contraction $\g_{(0)}$, see Theorem~\ref{thm:change-to-1} 
(resp.~\ref{thm:change2}) for the inner (resp.~outer) automorphisms of $\g$. Another ingredient is that if 
$\vth$ is inner and a certain label on $\eus K(\vth)$ is nonzero, then the $\vth$-contraction $\g_{(0)}$ is
isomorphic to a {\it parabolic contraction\/} of $\g$ (Theorem~\ref{thm:semidir}). The theory of parabolic 
contraction is developed in~\cite{sel}, and an interplay between two types of contractions enriches our 
knowledge of {\sf PC} subalgebras in both cases. For instance, we prove that $\gZ(\sln,\vth)$ is polynomial for any $\vth\in {\sf Int}^f(\sln)$ (Theorem~\ref{thm:svob-alg-sl}).

Frankly, we believe the equality $\ind\g_{(0)}=\ind\g$ holds for any $\vth\in {\sf Aut}^f(\g)$, and it is a 
challenge to prove it in full generality. This equality can be thought of as a $\vth$-generalisation of the 
{\it Elashvili conjecture}. For, a possible proof would require to check that, for a nilpotent element $x\in\g_1$, one has $\ind(\g^{x})_{(0)}=\ind \g^x$,
cf.~Corollary~\ref{cor:2}. 

We say that $\vth\in {\sf Aut}^f(\g)$ is $\gN$-{\it regular}, if $\g_1$ contains a regular nilpotent element 
of $\g$. Properties of the $\gN$-{regular} automorphisms are studied in~\cite[\S\,3]{p05}. In particular, if a 
connected component of ${\sf Aut}(\g)$ contains elements of order $m$, then it contains a unique 
$G$-orbit of $\gN$-regular elements of order $m$. That is, there are sufficiently many $\gN$-regular 
automorphisms of $\g$. We prove that a {\sf g.g.s.} exists for the $\gN$-regular $\vth$, see 
Theorem~\ref{thm:main1}. 
Furthermore, if $\vth$ and $\vth'$ belong to the same connected component of ${\sf Aut}(\g)$,
$|\vth|=|\vth'|$,  $\dim\g^\vth=\dim\g^{\vth'}$, and $\vth$ is $\gN$-regular, then
$\vth'$ also admits a {\sf g.g.s.}~(Theorem~\ref{thm:main3}). 

Another interesting feature is that if $\vth$ is inner and $\gN$-regular, then at most one label on 
$\eus K(\vth)$ can be bigger that $1$~(Theorem~\ref{thm:p_i=0,1}). Moreover, if $|\vth|$ does not exceed
the Coxeter number of $\g$, then all Kac labels belong to $\{0,1\}$.

\section{Preliminaries on {\sf PC} subalgebras and periodic automorphisms}
\label{sect:prelim}

\subsection{Compatible Poisson brackets}  \label{subs:compatible}
Let $\q$ be an arbitrary algebraic Lie algebra. The {\it index\/} of $\q$, $\ind\q$, is the minimal dimension 
of the stabilisers of $\xi\in\q^*$ with respect to the coadjoint representation of $\q$. If $\q$ is reductive, 
then $\ind\q=\rk\q$.
Two Poisson brackets are said to be {\it compatible} if their sum is again a Poisson bracket.
Suppose that $\{\ ,\,\}_t= \{\ ,\,\}'+ t\{\ ,\,\}''$, $t\in\BP^1$, is a pencil of compatible linear Poisson brackets on $\gS(\q)$, 
where $\BP^1=\bbk\cup\{\infty\}$ and $\{\ ,\,\}_1$ is the initial Lie--Poisson structure on $\q$.

Let $\q_{(t)}$ denote the Lie algebra structure on the vector space $\q$ corresponding to $\{\ ,\,\}_t$.
The function $(t\in \BP^1)\mapsto \ind \q_{(t)}$ is upper semi-continuous and therefore is constant on a
dense open subset of $\BP^1$. This subset is denoted by $\BP_{\sf reg}$, and we set
$\BP_{\sf sing}=\BP^1\setminus \BP_{\sf reg}$.  Then $\BP_{\sf sing}$ is finite and
\[
   t_0\in \BP_{\sf sing} \ \Longleftrightarrow \ \ind \q_{(t_0)}>\min_{t\in\BP^1}\ind \q_{(t)} .
\]
Let $\cz_{t}$ be the centre of the Poisson algebra $(\gS(\q), \{\ ,\,\}_t)$ and $\gZ$ the subalgebra
of $\gS(\q)$ generated by all $\cz_{t}$ with $t\in\BP_{\sf reg}$. We also write
\[
      \gZ={\sf alg}\lg \cz_t\mid t\in \BP_{\sf reg}\rg.
\]
Then $\gZ$ is Poisson commutative with respect to {\bf any} bracket $\{\ ,\,\}_t$ with $t\in\BP^1$. In cases to be treated below, $1\in \BP_{\sf reg}$ and all but finitely 
many algebras $\q_{(t)}$ are isomorphic to $\q$.  Then one can prove that such a $\gZ$ is a
{\sf PC} subalgebra of maximal transcendence degree in an appropriate class of subalgebras of
$\gS(\q)$, see~\cite{OY,bn-oy}.

\subsection{Periodic automorphisms of $\g$ and related {\sf PC} subalgebras of $\gS(\g)$}
\label{subs:theta}
Suppose that $\g$ is reductive and $\vth\in {\sf Aut}^f(\g)$. Using $\vth$, one can construct a pencil 
$\{\ ,\,\}_t= \{\ ,\,\}_{(0)}+ t\{\ ,\,\}_{(\infty)}$ of compatible
linear Poisson brackets on $\gS(\g)$, see \cite{py21} and Section~\ref{sect:index-g_0}. This pencil and
the related {\sf PC} subalgebra $\gZ=\gZ(\g,\vth)$ have the following properties:
\begin{itemize}
\item \ the Lie algebras $\g_{(t)}$, $t\in \bbk\setminus \{0\}$, are isomorphic to $\g$ and hence
$\BP_{\sf sing}\subset \{0,\infty\}$;
\item \ $\infty\in \BP_{\sf reg}$ if and only if $\g_0:=\g^\vth$ is abelian~\cite[Theorem\,3.2]{py21};
\item \ $\gZ(\g,\vth)\subset \gS(\g)^{\g_0}$~\cite[(3.6)]{py21}.
\end{itemize}
By~\cite[Prop.\,1.1]{m-y}, if $\gA$ is a {\sf PC} subalgebra of $\gS(\g)^{\g_0}$, then
\[
   \trdeg \gA \le \frac{1}{2}(\dim\g-\dim\g_0+\rk\g+\rk\g_0)=: \bb(\g,\vth) .
\]
If $\g_0$ is abelian, then the right-hand side becomes $(\dim\g+\rk\g)/2=:\bb(\g)$.
\\  \indent
Recall that $\gZ(\g,\vth)$ is generated by the centres $\cz_t$ with $t\in\BP_{\sf reg}$.

\begin{thm}[{\cite[Theorem\,3.10]{py21}}]    \label{thm:2.1}
If\/ $\ind\g_{(0)}=\ind\g$ (i.e., $0\in\BP_{\sf reg}$), then $\trdeg \gZ(\g,\vth) =\bb(\g,\vth)$.
\end{thm}
It is convenient to introduce the {\sf PC} subalgebra 
$\gZ_\times={\sf alg}\lg \cz_t\mid t\in \bbk\setminus \{0\}\rg \subset \gZ(\g,\vth)$, whose structure is easier 
to understand. Although $\gZ_\times$ can be a proper subalgebra of $\gZ(\g,\vth)$, this does not affect 
the transcendence degree, see \cite[Cor.\,3.8]{py21}. Moreover, there are many cases in which the centre 
$\cz_0$ can explicitly be described and one can check that $\cz_0\subset \gZ_\times$, see 
e.g.~\cite[Cor.\,4.7]{py21}. Then $\gZ(\g,\vth)$ is either equal to $\gZ_\times$  (if $\g_0$ is not abelian) or generated by $\gZ_\times$ and $\cz_\infty$  (if $\g_0$ is abelian). 

Another notion, which is useful in describing the structure of $\gZ_\times$, is that of a {\it good generating
system\/} in $\cz_1=\gS(\g)^\g$. As is well known, $\gS(\g)^\g$ is a polynomial algebra in $\rk\g$ 
generators. Let $H_1,\dots,H_l$ ($l=\rk\g$) be a set of algebraically independent homogeneous 
generators of $\gS(\g)^\g$ such that each $H_i$ is a $\vth$-eigenvector. Then we say that
$H_1,\dots,H_l$ is a set of $\vth$-{\it generators\/} in $\gS(\g)^\g$. If $|\vth|=m$ and 
$\g=\bigoplus_{i=0}^{m-1}\g_i$ is the associated $\BZ_m$-grading, then we consider the 1-parameter 
group $\vp: \bbk^*\to \GL(\g)$ such that $\vp(t){\cdot}x=t^ix$ for $x\in\g_i$. (Note that $\vp(\zeta)=\vth$.) 
This yields the natural $\BZ$-grading in $\gS(\g)$. If $\vp(t){\cdot}H_j=\sum_i t^i H_{j,i}$, then the nonzero 
polynomials $H_{j,i}$ are called the $\vp$-{\it homogeneous} (or {\it bi-homogeneous}) components of 
$H_j$. We say that $i$ is the $\vp$-{\it degree\/} of $H_{j,i}$. Let $H_j^\bullet$ denote the 
$\vp$-homogeneous component of $H_j$ of the maximal $\vp$-degree. This maximal $\vp$-degree is 
denoted by $\deg_\vp(H_j)$.

\begin{df}   \label{df:ggs}
A set of $\vth$-generators $H_1,\dots,H_l\in \gS(\g)^\g$ is called a {\it good generating system} 
(={\sf g.g.s.}) with respect to $\vth$, if $H_1^\bullet,\dots,H_l^\bullet$ are algebraically independent. If
there is {\sf g.g.s.} with respect to $\vth$, we also say that $\vth$ {\it admits\/} a {\sf g.g.s.}
\end{df}

The following is the main tool for checking that a set of $\vth$-generators forms a {\sf g.g.s.}

\begin{thm}[{\cite[Theorem\,3.8]{contr}}]    \label{thm:kokosik}
Let $H_1,\dots,H_l$ be a set of $\vth$-generators in $\gS(\g)^\g$. Then
\begin{itemize}
\item $\sum_{i=1}^l \deg_\vp H_j \ge  \sum_{i=1}^{m-1}i\dim\g_i =: D_\vth$;
\item $H_1,\dots,H_l$ is a {\sf g.g.s.} if and only if \ $\sum_{i=1}^l \deg_\vp H_j = D_\vth$.
\end{itemize}
 \end{thm}
 
By Theorems~4.3 \& 4.6 in \cite{py21}, we have
\begin{thm}        \label{thm:2.3}
If\/ $\ind\g_{(0)}=l$ and $H_1,\dots,H_l$ is {\sf g.g.s.} with respect to $\vth$, then $\gZ_\times$ is a 
polynomial algebra, which is freely generated by the $\vp$-homogeneous components of
$H_1,\dots,H_l$.
\end{thm}
Theorems \ref{thm:2.1} and \ref{thm:2.3} imply that under these hypotheses the total number of the 
nonzero bi-homogeneous components of all generators $H_j$ equals $\bb(\g,\vth)$. 

\subsection{The Kac diagram of $\vth\in{\sf \Aut}^f(\g)$} 
\label{subs:kac}
A pair $(\g,\vth)$ is {\it decomposable}, if $\g$ is a direct sum of non-trivial
$\vth$-stable ideals. Otherwise  $(\g,\vth)$ is said to be {\it indecomposable}. A classification of finite order 
automorphisms readily reduces to the indecomposable case. The centre of $\g$ is always a $\vth$-stable ideal and automorphisms of an abelian Lie algebra have no particular significance (in our context). Therefore, assume that $\g$ is semisimple. 

If $\g$ is not simple and $(\g,\vth)$ is  indecomposable, then $\g=\h^{\oplus n}$ is a sum of $n$ copies of 
a simple Lie algebra $\h$ and $\vth$ is a composition of a periodic automorphism of $\h$ and a cyclic 
permutation of the summands. 

Below we assume that $\g$ is simple. By a result of R.\,Steinberg~\cite[Theorem\,7.5]{st},
every semisimple automorphism of $\g$ fixes a Borel subalgebra of $\g$ and a Cartan 
subalgebra thereof. Let $\be$ be a $\vth$-stable Borel subalgebra and $\te\subset\be$ a 
$\vth$-stable Cartan subalgebra. This yields a $\vth$-stable triangular decomposition 
$\g=\ut^-\oplus\te\oplus\ut$, where $\ut=[\be,\be]$. Let $\Delta=\Delta(\g)$ be the set of roots of $\te$, 
$\Delta^+$ the set of positive roots corresponding to $\ut$, and $\Pi\subset\Delta^+$ the set of simple 
roots. Let $\g^\gamma$ be the root space for $\gamma\in\Delta$. Hence $\ut=\bigoplus_{\gamma\in\Delta^+}\g^\gamma$.

Clearly, $\vth$ induces a permutation of $\Pi$, which is an automorphism of the Dynkin diagram,
and $\vth$ is inner if and only if this permutation is trivial. Accordingly, $\vth$ can be written as a product 
$\sigma{\cdot}\vth'$, where $\vth'$ is inner and $\sigma$ is the so-called {\it diagram automorphism\/} of 
$\g$. We refer to \cite[\S\,8.2]{kac} for an explicit construction and properties of $\sigma$. In particular, 
$\sigma$ depends only on the connected component of ${\sf Aut}(\g)$ that contains $\vth$ and $\Ord$ 
equals the order of the corresponding permutation of $\Pi$. The {\it index\/} of $\vth\in {\sf Aut}^f(\g)$ is 
the order of the image of $\vth$ in ${\sf Aut}(\g)/{\sf Int}(\g)$, i.e., the order of the corresponding diagram 
automorphism.

\subsubsection{The inner periodic automorphisms}
\label{subsub:int}
Set $\Pi=\{\ap_1,\ldots,\ap_l\}$ and let $\delta=\sum_{i=1}^l n_i \ap_i$ be the highest root in $\Delta^+$. 
An inner periodic automorphism with $\te\subset\g_0$ is determined by an $(l+1)$-tuple of non-negative 
integers ({\it Kac labels})  $\boldsymbol{p}=(p_0,p_1,\ldots,p_l)$ such that $\gcd(p_0,\ldots,p_l)=1$ and 
$\boldsymbol{p}\ne (0,\dots,0)$.
Set $m:=p_0+\sum_{i=1}^{l} n_i p_i$ and let $\ov{p_i}$ denote the unique representative of
$\{0,1,\dots,m-1\}$ such that $p_i \equiv \ov{p_i} \pmod m$. The $\BZ_m$-grading 
$\g=\bigoplus_{i=0}^{m-1}\g_i$ corresponding to $\vth=\vth(\boldsymbol{p})$ is defined 
by the conditions that
\[
   \g^{\ap_i}\subset\g_{\ov{p_i}} \ \text{ for $i=1,\dots,l$}, 
   \enskip \g^{-\delta}\subset\g_{\ov{p_0}},  \text{ and } \ \te\subset\g_0.
\]
For our purposes, it is better to introduce first the $\BZ$-grading of $\g$ defined by $(p_1,\ldots,p_l)$ 
and then factorise ("glue") it modulo $m$, see Section~\ref{sect:inner-auto} for details.

The {\it Kac diagram\/} $\eus K(\vth)$ of $\vth=\vth(\boldsymbol{p})$ is the {\bf affine} (=\,extended) 
Dynkin diagram of $\g$, $\tilde{\eus D}(\g)$, equipped with the labels $p_0,p_1,\dots,p_l$. In 
$\eus K(\vth)$, the $i$-th node of the usual Dynkin diagram ${\eus D}(\g)$ represents $\ap_i$ and the 
extra node represents $-\delta$. It is convenient to assume that $\ap_0=-\delta$ and $n_0=1$. Then 
$(l+1)$-tuple $(n_0,n_1,\dots,n_l)$ yields coefficients of linear dependence for $\ap_0,\ap_1,\dots,\ap_l$. 
Set $\widehat\Pi=\Pi\cup\{\ap_0\}$. If $n_i=1$ for $i\ge 1$, then the subdiagram without the $i$-th node is 
isomorphic to ${\eus D}(\g)$ and $\widehat\Pi\setminus \{\ap_i\}$ is another set of simple roots in 
$\Delta$. Hence any node of $\tilde{\eus D}(\g)$ with $n_i=1$ can be regarded as an extra node, which 
merely corresponds to another choice of a Borel subalgebra containing our fixed Cartan subalgebra 
$\te$. Practically this means that we consider these Kac diagrams modulo the action of the 
automorphism group of the graph $\tilde{\eus D}(\g)$.

\subsubsection{The outer periodic automorphisms}
\label{subsub:out}
Let $\sigma$ be the diagram automorphism of $\g$ related to $\vth$. The orders of nontrivial 
diagram automorphisms are:
\begin{itemize}
\item \ $\GR{A}{n}$ ($n\ge2$), \ $\GR{D}{n}$ ($n\ge4$), \ $\GR{E}{6}$: \ $\Ord=2$;
\item \ $\GR{D}{4}$: \ $\Ord=3$.
\end{itemize}

\noindent
Therefore, $\sigma$ defines either $\BZ_2$- or $\BZ_3$-grading of $\g$. To avoid confusion with 
the $\vth$-grading, this $\sigma$-grading is denoted as follows:
\beq      \label{eq:sigma-grad}
\g=  \begin{cases}  \g^{(\sigma)}_0\oplus\g_1^{(\sigma)}, &\text{ if } \Ord=2 ;\\
         \g^{(\sigma)}_0\oplus\g_1^{(\sigma)}\oplus\g^{(\sigma)}_2, & \text{ if } \Ord=3 , \end{cases}
\eeq    
and the latter occurs only for $\g=\gt{so}_8$. In all cases, $\g^\sigma=\g^{(\sigma)}_0$ is a simple Lie 
algebra and each $\g^{(\sigma)}_i$ is a simple $\g^{\sigma}$-module. If $\Ord=3$, then 
$\g_1^{(\sigma)}\simeq\g_2^{(\sigma)}$ as $\g^{\sigma}$-modules and 
$\g_2^{(\sigma)}=[\g_1^{(\sigma)},\g_1^{(\sigma)}]$. Since $\be$ and $\te$ are $\sigma$-stable, 
$\be^\sigma=\te^\sigma\oplus\ut^\sigma$ is a Borel subalgebra of $\g^\sigma$ and $\te_0=\te^\sigma$ is 
a Cartan subalgebra of both $\g^\sigma$ 
and $\g_0=\g^\vth$. Let $\Delta^+(\g^\sigma)$ be the set of positive roots of $\g^\sigma$ corresponding 
to $\ut^\sigma$ and let $\{\nu_1,\dots,\nu_r\}$ be the set of simple 
roots in $\Delta^+(\g^\sigma)$.  

The Kac diagrams of  outer periodic automorphism are supported on the twisted affine Dynkin diagrams
of index 2 and 3, see \cite[\S\,8]{vi76} and~\cite[Table\,3]{t41}. Such a diagram has $r+1$ nodes, where 
$r=\rk\g^\sigma$, certain $r$ nodes comprise the Dynkin diagram of the simple Lie algebra $\g^{\sigma}$, 
and the additional node represents the lowest weight $-\delta_1$ of the $\g^\sigma$-module 
$\g_1^{(\sigma)}$. Write $\delta_1=\sum_{i=1}^r a_i' \nu_i$ and set $a'_0=1$. Then the $(r+1)$-tuple 
$(a'_0,a'_1,\dots,a'_r)$  yields coefficients of linear dependence for $-\delta_1,\nu_1,\dots,\nu_r$.

The subalgebras $\g^\sigma$ and $\g^\sigma$-module $\g_1^{(\sigma)}$ are gathered in the following 
table, where $\VV_\lb$ is a simple $\g^\sigma$-module with highest weight $\lb$, and the numbering of 
simple roots and fundamental weights $\{\vp_i\}$ for  $\g^\sigma$ follows~\cite[Table\,1]{t41}.

\begin{center}
\begin{tabular}{|c|lllll|}  \hline
$\g$ & $\GR{A}{2r}$ & $\GR{A}{2r-1}$ & $\GR{D}{r+1}$ & $\GR{E}{6}$ & $\GR{D}{4}$ \\
$\g^\sigma$ & $\GR{B}{r}$ & $\GR{C}{r}$  &  $\GR{B}{r}$ & $\GR{F}{4}$ & $\GR{G}{2}$ \\
$\g^{(\sigma)}_1$ & $\VV_{2\vp_1}$ & $\VV_{\vp_2}$ & $\VV_{\vp_1}$ & $\VV_{\vp_1}$ & 
$\VV_{\vp_1}$  \\ \hline
twisted diagram \rule{0pt}{3ex} & $\GR{A}{2r}^{(2)}$ & $\GR{A}{2r-1}^{(2)}$ & $\GR{D}{r+1}^{(2)}$ & $\GR{E}{6}^{(2)}$ & 
$\GR{D}{4}^{(3)}$ 
\\  \hline
\end{tabular}
\end{center}

\vskip1ex\noindent
Some of the twisted affine diagrams are depicted below. We enhance these 
diagrams with the coefficients $\{a'_i\}$ over the nodes and the corresponding roots under the nodes.

\centerline{
$\GR{A}{2}^{(2)}$: \quad 
\raisebox{-2ex}{
\begin{picture}(55,36)(-3,-5)
\setlength{\unitlength}{0.018in} 
\multiput(10,8)(20,0){2}{\circle{6}}
\put(20.32,5.75){$>$} 
\multiput(12.86,6.5)(0,3){2}{\line(1,0){10.6}}    
\multiput(12.86,7.5)(0,1){2}{\line(1,0){11.5}} 
\put(7.5,14.2){\footnotesize $1$} \put(27.5,14.2){\footnotesize $2$}
 \put(2,-4){\footnotesize $-\delta_1$} \put(27,-4){\footnotesize $\nu_1$}
 \put(38,6){;}
\end{picture} }
\qquad
$\GR{A}{2r}^{(2)}$, $r\ge 2$: \quad 
\raisebox{-2ex}{
\begin{picture}(115,36)(75,-5)
\setlength{\unitlength}{0.018in} 
\multiput(70,8)(20,0){3}{\circle{6}}
\multiput(150,8)(20,0){2}{\circle{6}}
\multiput(72.86,7)(0,2){2}{\line(1,0){10.8}}    \put(80.24,5.72){$>$} 
\multiput(152.86,7)(0,2){2}{\line(1,0){10.8}}      \put(160.24,5.72){$>$}
\put(93,8){\line(1,0){14}}
\multiput(113,8)(27,0){2}{\line(1,0){7}}
\put(124,5){$\cdots$}
\put(67.5,14.2){\footnotesize $1$}  \put(87.5,14.2){\footnotesize $2$} \put(107.5,14.2){\footnotesize $2$}
\put(147.5,14.2){\footnotesize $2$} \put(167.5,14.2){\footnotesize $2$}
\put(62,-4){\footnotesize $-\delta_1$}
\put(86.5,-4){\footnotesize $\nu_1$} \put(106.5,-4){\footnotesize $\nu_2$}
\put(143.5,-4){\footnotesize $\nu_{r{-}1}$} \put(166.5,-4){\footnotesize $\nu_{r}$}
\put(177,6){;}
\end{picture} }
}

\begin{center} 
$\GR{E}{6}^{(2)}$: \quad  
\raisebox{-1.7ex}{ \begin{picture}(116,35)(-23,-3)
\setlength{\unitlength}{0.018in} 
\multiput(-10,8)(20,0){5}{\circle{6}}
\put(7,-4){\footnotesize $\nu_1$} \put(27,-4){\footnotesize $\nu_2$}
\put(47,-4){\footnotesize $\nu_3$} \put(67,-4){\footnotesize $\nu_4$}
\put(-19,-4){\footnotesize $-\delta_1$}
\put(-7,8){\line(1,0){14}}
\put(13,8){\line(1,0){14}} \put(53,8){\line(1,0){14}}
\multiput(36.34,7)(0,2){2}{\line(1,0){10.8}}
\put(32.8,5.72){$<$} 
\put(-12.5,14.2){\footnotesize $1$}  \put(7.5,14.2){\footnotesize $2$} \put(27.5,14.2){\footnotesize $3$}
\put(47.5,14.2){\footnotesize $2$} \put(67.5,14.2){\footnotesize $1$}
\put(77,6){;}
\end{picture} }
\qquad
$\GR{D}{4}^{(3)}$: \quad  
\raisebox{-1.7ex}{\begin{picture}(56,35)(-23,-3)
\setlength{\unitlength}{0.018in} 
\multiput(-10,8)(20,0){3}{\circle{6}}
\put(7,-4){\footnotesize $\nu_1$} \put(27,-4){\footnotesize $\nu_2$}
\put(-19,-4){\footnotesize $-\delta_1$}
\put(-7,8){\line(1,0){14}}
\put(13,8){\line(1,0){14}}
\multiput(16.34,6.6)(0,2.8){2}{\line(1,0){11.16}}
\put(12.8,5.72){$<$} 
\put(-12.5,14.2){\footnotesize $1$}  \put(7.5,14.2){\footnotesize $2$} \put(27.5,14.2){\footnotesize $1$}
\put(37,6){.}
\end{picture} }
\end{center} 

\vskip1ex\noindent
Let $\boldsymbol{p}=(p_0, p_1,\dots, p_r)$ be an $(r+1)$-tuple such that 
$\boldsymbol{p}\ne (0,0,\dots,0)$ and $\gcd(p_0, p_1\dots, p_r)=1$. The Kac diagram of 
$\vth=\vth(\boldsymbol{p})$ is the required twisted affine diagram equipped with the labels 
$(p_0, p_1,\dots, p_r)$ over the nodes. Then $m=|\vth(\boldsymbol{p})|=\Ord{\cdot} \sum_{i=0}^r a_i' p_i$. 
\\ \indent
Similar to the inner case, the $\BZ_m$-grading $\g=\bigoplus_{i=0}^{m-1}\g_i$ corresponding to 
$\vth=\vth(\boldsymbol{p})$ is defined by the conditions that
\[
   (\g^{\sigma})^{\nu_i}\subset\g_{\ov{p_i}} \ \text{ for $i=1,\dots,r$}, 
   \enskip (\g_1^{(\sigma)})^{-\delta_1}\subset\g_{\ov{p_0}},  \text{ and } \ \te^\sigma\subset\g_0.
\]
In Section~\ref{sect:outer}, we give a detailed description of this $\BZ_m$-grading  
and use it to prove a modification result on $\eus K(\vth)$ and the structure of $\g_{(0)}$.

\subsection{The description of $\g_0$ and $\g_1$ via the Kac diagram of $\vth$}
\label{subs:g_0}
Let $p_0,p_1,\dots,p_l$ be the Kac labels of $\vth\in {\sf Int}^f(\g)$. Then the subdiagram of nodes in 
$\tilde{\eus D}(\g)$ such that $p_i=0$ is the Dynkin diagram of $[\g_0,\g_0]$, while the dimension of the 
centre of $\g_0$ equals $\#\{i\mid p_i\ne 0\}-1$. Then $\{\ap_i\mid i\in\{0,1,\dots,l\} \ \& \ p_i=1\}$ are the 
lowest weights of the simple $\g_0$-modules in $\g_1$, i.e., if $\VV^-_{\mu}$ stands for the 
$\g_0$-module with {\sl lowest\/} weight $\mu$, then
\[
    \g_1=\bigoplus_{i:\ p_i=1} \VV^-_{\ap_i} .
\]
The same principle applies to the outer periodic automorphisms, $\tilde{\eus D}(\g)$ being replaced
with the respective twisted affine Dynkin diagram. These results are 
contained in~\cite[Prop.\,17]{vi76}.

It follows  that the subalgebra of $\vth$-fixed points, $\g_0$, is semisimple if and only 
if $\eus K(\vth)$ has a unique nonzero label. At the other extreme, $\g_0$ is abelian if and only if all $p_i$ are nonzero. Furthermore, if all $p_i\le 1$, then the following conditions are equivalent:
\begin{itemize}
\item \ $\g_0=\g^\vth$ is semisimple;
\item \ $\g_1$ is a simple $\g_0$-module;
\item \ $\eus K(\vth)$ has a unique nonzero label.
\end{itemize}

\begin{ex}
Take the automorphism of $\GR{D}{4}$ of index $3$ with Kac labels $p_0=p_2=1, p_1=0$, 
i.e., $\eus K(\vth)$ is 
\quad  
\raisebox{-1.5ex}{\begin{picture}(65,34)(-18,0)
\setlength{\unitlength}{0.018in} 
\multiput(-10,8)(20,0){3}{\circle{6}}
\put(-7,8){\line(1,0){14}}
\put(13,8){\line(1,0){14}}
\multiput(16.34,6.6)(0,2.8){2}{\line(1,0){11.16}}
\put(12.8,5.72){$<$} 
\put(-12.5,14.2){\footnotesize {\it\bfseries 1}}  \put(7.5,14.2){\footnotesize {\it\bfseries 0}} 
\put(27.5,14.2){\footnotesize {\it\bfseries 1}}
\put(35,6){.}
\end{picture} } Then $|\vth|=3(1+1)=6$, $G_0=\SL_2\times T_1$, and 
$\g_1=\VV_{\vp}{\cdot}\esi+\VV_{3\vp}{\cdot}\esi^{-1}$ as $G_0$-module. Here $\vp$ is the fundamental weight of $\SL_2$ and $\esi$ is the basic character of $T_1$.
\end{ex}

\section{On the index of periodic contractions of semisimple Lie algebras}
\label{sect:index-g_0}

\noindent
In this section, we recall the structure of Lie algebras $\g_{(0)}$ and $\g_{(\infty)}$ and then prove that
$\ind \g_{(0)}=\ind\g$ for small values of $m$.
Let $\zeta=\sqrt[m]1$ be a fixed primitive root of unity.  Then 
\beq         \label{eq:grading}
     \g=\bigoplus_{i=0}^{m-1} \g_i ,
\eeq
where the eigenvalue of $\vth$ on $\g_i$ is $\zeta^i$. The Lie algebras $\g$, $\g_{(0)}$, and 
$\g_{(\infty)}$ have the same underlying vector space, but different Lie brackets, denoted 
$[\ ,\,]$, $[\ ,\,]_{(0)}$, and $[\ ,\,]_{(\infty)}$, respectively. More precisely, 
\beq         \label{eq:g-and-g_0}
\begin{split}  
 &\text{ if  $i+j\le m-1$, then $[\g_i,\g_j]=[\g_i,\g_j]_{(0)} \subset \g_{i+j}$;} \\
 &\text{ if  $i+j > m-1$, then $[\g_i,\g_j]_{(0)}=0$, while $[\g_i,\g_j] \subset \g_{i+j-m}$.}
\end{split}
\eeq
Hence vector space decomposition~\eqref{eq:grading} is a $\BZ_m$-grading for $\g$, but it is
an $\BN$-grading for $\g_{(0)}$. Then the $(\infty)$-bracket can be defined as
\[
   [\ ,\,]_{(\infty)}=[\ ,\,] - [\ ,\,]_{(0)} .
\]
One readily verifies that $\g_{(\infty)}$ is also $\BN$-graded and its component of grade $i$ is
$\g_{m-i}$ for $i=1,2,\dots,m$; in particular, the component of grade $0$ is trivial.
This implies that $\g_{(\infty)}$ is nilpotent, cf.~\cite[Prop.\,2.3]{py21}.

Since $\ind\g_{(\infty)}$ is known~\cite[Theorem\,3.2]{py21}, we are interested now in the 
problem of computing $\ind\g_{(0)}$. Let us recall some relevant results. 
\begin{itemize}
\item By the semi-continuity of index under contractions, one has $\ind  \g_{(0)}\ge \ind\g$;
\item if $m=2$, then the $\BZ_2$-contraction $\g_{(0)}\simeq\g_0\ltimes\g_1^{\sf ab}$ is a semi-direct product and therefore
$\ind \g_{(0)}=\ind\g$~\cite[Prop.\,2.9]{p07};
\item if $\g_1$ contains a regular element of $\g$, 
then $\ind \g_{(0)}=\ind\g$~\cite[Prop.\,5.3]{p09}. 
\end{itemize}

\begin{conj}   \label{conj1}
For any periodic automorphism $\vth$, one has $\ind \g_{(0)}=\ind\g$.
\end{conj}

Let us record the following simple fact.
\begin{lm}      \label{lm:ss}
It suffices to verify Conjecture~\ref{conj1} for the semisimple Lie algebras. 
\end{lm}
\begin{proof}
Write $\g=\es\oplus\ce$, where $\ce$ is the centre of $\g$ and $\es=[\g,\g]$. Then $\g_{(0)}=\es_{(0)}\oplus\ce_{(0)}$. Since $\ce$ is an Abelian Lie algebra, then so is $\ce_{(0)}$ and 
$\ind\ce=\ind\ce_{(0)}$. The result follows.
\end{proof}

\begin{lm}      \label{lm:1}
Suppose that\/  $\ind(\g_{(0)})^\xi=\ind\g$ \ for some $\xi\in\g_{(0)}^*$. Then\/ $\ind\g_{(0)}=\ind\g$.
\end{lm}
\begin{proof}
By Vinberg's inequality for $\g_{(0)}$ (cf.~\cite[Prop.\,1.6 \& Cor.\,1.7]{p03}) and semi-continuity of index, one has 
\[
              \ind(\g_{(0)})^\xi\ge\ind\g_{(0)}\ge \ind\g . \qedhere
\]
\end{proof}
\noindent
The Killing form $\kappa$ on $\g$ induces the isomorphism $\tau:\g\to\g^*$ with 
$\tau(x)(y):=\kappa(x,y)$ for all $x,y\in\g$.
Clearly $\tau$ restricts to an isomorphism 
$\g_i\simeq \g^*_{m-i}$ for each $i$. 
Set $\xi_x:=\tau(x)$.
Having identified $\g^*$ and $\g_{(0)}^*$ as vector spaces, we may regard $\xi_x$ as an element of
$\g_{(0)}^*$. Then $(\g_{(0)})^{\xi_x}$ denotes the stabiliser of $\xi_x$ with respect to the coadjoint representation of $\g_{(0)}$.

\begin{prop}   \label{prop:1}
Let $x\in\g_1\subset\g$ be arbitrary.
\begin{itemize}
\item[\sf (i)] \ Upon the identification of $\g$ and $\g_{(0)}$,  the vector spaces $\g^x$ and 
$(\g_{(0)})^{\xi_x}$ coincide. 
\item[\sf (ii)] \ Moreover, the Lie algebra $\g^x$ is $\vth$-stable and its $\vth$-contraction
$(\g^x)_{(0)}$ is isomorphic to $(\g_{(0)})^{\xi_x}$ as a Lie algebra.
\end{itemize}
\end{prop}
\begin{proof}
{\sf (i)} \ Since the Lie algebra $\g_{(0)}$ is $\BN$-graded, $(\g_{(0)})^{\xi_x}$ is $\BN$-graded as well. 
On the other hand, $\g^x$ inherits the $\BZ_m$-grading from $\g$. Let us show that
the vector spaces $\g^x\cap\g_i$ and $(\g_{(0)})^{\xi_x}\cap\g_i$ are equal for each $i$.
Let $\ad^*_{(0)}$ denote the coadjoint representation of $\g_{(0)}$. For $y\in\g_j$, we have 

$[x,y]\in\begin{cases}\g_{j+1}, & 0\le j\le m{-}2 \\ \g_0, & j=m-1 \end{cases}$ \quad and \  
$\ad^*_{(0)}(y)(\xi_x)\in \g^*_{m-1-j}$ for $j=0,1,\dots,m{-}1$. \\ 
For any $j$, we then obtain
\[
 \ad_0^*(y)\xi_x=0 \ \Longleftrightarrow \  \xi_x([y,\g_{m-1-j}])=0 \  \Longleftrightarrow \ 
 \kappa([x,y], \g_{m-1-j})=0 \  \Longleftrightarrow \ [x,y]=0.
\]
This proves {\sf (i)}. 

{\sf (ii)} This follows from {\sf (i)} and the general relationship between the Lie brackets of the initial Lie 
algebra and a $\BZ_m$-contraction of it, cf.~\eqref{eq:g-and-g_0}.
\end{proof}

\begin{cl}      \label{cor:2}
If there is an $x\in \g_{1}$ such that\/ $\ind (\g^{x})_{(0)}=\ind\g^{x}$, then\/ $\ind\g_{(0)}=\ind\g$.
\end{cl}
\begin{proof}
One has  $\ind(\g_{(0)})^{\xi_x}=\ind(\g^{x})_{(0)}=\ind\g^{x}=\ind\g$, where the last equality is the celebrated {\sl Elashvili conjecture} proved via contributions of many people, see~\cite{CM}. Then Lemma~\ref{lm:1} applies.
\end{proof}

These results yield the {\it\bfseries induction step} for computing $\ind\g_{(0)}$. If $\g$ is semisimple and 
$x\in\g_1$ is a nonzero semisimple element, then $\g^x\subsetneq \g$, $\g^x$ is reductive, 
$\ind\g^x=\ind\g$, and $\vth$ preserves $\g^x$. Hence it suffices to verify Conjecture~\ref{conj1} for the 
smaller semisimple Lie algebra $[\g^x,\g^x]$. One can perform such a step as long as $\g_1$ contains 
semisimple elements. The base of induction is the case in which $\g_1$ contains 
no nonzero semisimple elements. Then the existence of the Jordan decomposition in $\g_1$~\cite[\S\,1.4]{vi76} 
implies that all elements of $\g_1$ are nilpotent. Actually,  the `base' can be achieved in just one step.
Recall from \cite{vi76} that a {\it Cartan subspace\/} of $\g_1$ is a maximal subspace $\ce$ consisting of
pairwise commuting semisimple elements. By~\cite[\S\,3,4]{vi76}, all Cartan subspaces are 
$G_0$-conjugate and $\dim\ce=\dim \g_1\md G_0$. The number $\dim\ce$ is called the {\it rank\/} of 
$(\g,\vth,m)$. We also denote it by $\rk\!(\g_0,\g_1)$. If $x\in \ce$ is a generic element, then 
$\es=[\g^x,\g^x]$ has the property that $\es_1$ consists of nilpotent elements.

Thus, in order to confirm Conjecture~\ref{conj1}, one should be able to handle the automorphisms 
$\vth$ of semisimple Lie algebras $\g$ such that $\g_1\subset\N$.
Using previous results, we can do it now for $m=3$ and for $m=4,5$ (with some reservations, see Proposition~\ref{prop:3}).

\begin{prop}  \label{prop:2}
If $m=3$, then $\ind\g_{(0)}=\ind\g$.
\end{prop}
\begin{proof}
By the inductive procedure above, we may assume that $\g_1\subset\N$. Then $G_0$ has finitely many 
orbits in $\g_1$~\cite[\S\,2.3]{vi76}. Take $x\in\g_1$ from the dense $G_0$-orbit. Then $[\g_0,x]=\g_1$ 
and hence $\g^{x}$ has the trivial projection to $\g_2$, i.e., $\g^{x}=\g_0^{x}\oplus\g_1^{x}$. This implies 
that $[\g_1^{x},\g_1^{x}]=0$ and therefore the Lie algebras $\g^{x}$ and $\g^{x}_{(0)}$ are isomorphic. 
Since $\ind\g^x=\ind\g$ by the {\sl Elashvili conjecture},  the assertion follows from Corollary~\ref{cor:2}.
\end{proof}

Recall that the action of a reductive group $H$ on an irreducible affine variety $X$ is {\it stable}, if the 
union of all closed $H$-orbits is dense in $X$. For $x\in\g_1=X$ and $H=G_0$, the orbit  $G_0{\cdot}x$ 
is closed if and only if $x$ is semisimple in $\g$~\cite[\S\,2.4]{vi76}. Therefore, the linear action of $G_0$ 
on $\g_1$ is stable if and only if the subset of semisimple elements of $\g$ is dense in $\g_1$. 

\begin{prop}       \label{prop:3}
Suppose that $m=4,5$ and the action $(G_0:\g_1)$ is stable. Then $\ind\g_{(0)}=\ind\g$.
\end{prop}
\begin{proof}
If $x\in\g_1$ is semisimple, then the action $(G^x_0:\g^x_1)$ is again stable. Therefore, for a generic 
semisimple $x\in\ce\subset \g_1$, the induction step
provides the semisimple Lie algebra $\es=[\g^x,\g^x]$ such that $\es_1=0$. Then $\es_{m-1}=0$ as well.

\un{$m=4$:} \ 
Here $\es=\es_0\oplus\es_2$ and  $\vth|_\es$ is of order $2$. Therefore, 
$\es_{(0)}=\es_0\rtimes \es_2^{\sf ab}$ is a $\BZ_2$-contraction of $\es$ and hence $\ind\es_{(0)}=\ind\es$.

\un{$m=5$:} \ Now $\es=\es_0\oplus\es_2\oplus\es_3$ and $\vth\vert_\es$ is still of order 5 (if
$\es_2\oplus\es_3\ne 0$). The absence of $\es_1$ and $\es_4$ implies that
$[\es_2\oplus\es_3,\es_2\oplus\es_3]\subset \es_0$, i.e., $\es$ can be regarded as $\BZ_2$-graded 
algebra. Thus, by~\eqref{eq:g-and-g_0}, $\es_{(0)}\simeq\es_0\rtimes (\es_2\oplus\es_3)^{\sf ab}$ is again a 
$\BZ_2$-contraction and hence $\ind\es_{(0)}=\ind\es$.
\end{proof}

\begin{ex}              \label{ex:F4}
For $\g$ of type $\GR{F}{4}$, the affine Dynkin diagram is
\quad  
\raisebox{-2ex}{ \begin{picture}(118,35)(-12,-5)
\setlength{\unitlength}{0.018in} 
\multiput(-10,8)(20,0){5}{\circle{6}}
\put(-13,-3.8){\footnotesize $\ap_1$} \put(7,-3.8){\footnotesize $\ap_2$} \put(27,-3.8){\footnotesize $\ap_3$}
\put(47,-3.8){\footnotesize $\ap_4$} \put(67,-3.8){\footnotesize $\ap_0$}
\put(-7,8){\line(1,0){14}}
\put(33,8){\line(1,0){14}} \put(53,8){\line(1,0){14}}
\multiput(16.34,7)(0,2){2}{\line(1,0){10.8}}
\put(12.8,5.72){$<$} 
\put(-12.5,14.2){\footnotesize $2$}  \put(7.5,14.2){\footnotesize $4$} \put(27.5,14.2){\footnotesize $3$}
\put(47.5,14.2){\footnotesize $2$} \put(67.5,14.2){\footnotesize $1$}
\put(75,6){.}
\end{picture} }
Take $\vth$ with the following Kac diagram

\begin{center}
$\eus K(\vth)$: \quad  
\raisebox{-2ex}{ \begin{picture}(116,34)(-23,-5)
\setlength{\unitlength}{0.018in} 
\multiput(-10,8)(20,0){5}{\circle{6}}
\put(-13,-4){\footnotesize $\vp'$} 
\put(27,-4){\footnotesize $\vp_1$}
\put(47,-4){\footnotesize $\vp_2$} \put(67,-4){\footnotesize $\vp_3$}
\put(-7,8){\line(1,0){14}}
\put(33,8){\line(1,0){14}} \put(53,8){\line(1,0){14}}
\multiput(16.34,7)(0,2){2}{\line(1,0){10.8}}
\put(12.8,5.72){$<$} 
\put(-12.5,14.2){\footnotesize {\it\bfseries 0}}  \put(7.5,14.2){\footnotesize {\it\bfseries 1}} \put(27.5,14.2){\footnotesize  {\it\bfseries 0}}
\put(47.5,14.2){\footnotesize  {\it\bfseries 0}} \put(67.5,14.2){\footnotesize  {\it\bfseries 0}}
\put(77,6){.}
\end{picture} }
\end{center}
\vskip1ex
\noindent
Then $|\vth|=4$, $\g_0=\GR{A}{3}\times \GR{A}{1}$, and $\g_1=\VV_{\vp_3}\otimes\VV_{\vp'}$ (or
$\g_1=\vp_3\vp'$) as a $\g_0$-module.  For the reader's convenience, we also provide the (numbering of the) fundamental weights of $\g_0$.
Since $G_0$ has a dense orbit in $\g_1$, we have $\g_1\subset \N$ and the induction step does not 
apply. Actually, our methods, including those developed in Section~\ref{sect:inner-auto}, do not work 
here, and the exact value of $\ind\g_{(0)}$ is not known yet.
\end{ex}

\section{Inner automorphisms, $\BZ$-gradings, and parabolic contractions of $\g$} 
\label{sect:inner-auto}

\noindent
In this section, we prove that, for {\bf certain} $\vth\in{\sf Int}^f(\g)$, the $\vth$-contraction $\g_{(0)}$ is 
isomorphic to a parabolic contraction of $\g$. Then comparing the results obtained earlier for parabolic 
contractions~\cite{sel} and $\vth$-contractions~\cite{py21} yields new knowledge in both instances.
 
First, we need an explicit description of $\vth\in {\sf Int}^f(\g)$ via a $\BZ$-grading of $\g$ 
associated with the Kac diagram $\eus K(\vth)$. Recall that $\eus K(\vth)$ is the affine Dynkin diagram of $\g$, equipped with numerical labels $p_0,p_1,\dots,p_l$, where $p_0$ is the label at the extra node.

As in Section~\ref{subs:kac},  $\l=\rk\g$, $\Pi=\{\ap_1,\dots,\ap_l\}$, 
$\delta=\sum_{i=1}^l n_i \ap_i\in\Delta^+$ is the highest root, $n_0=1$, 
and $m= |\vth|=\sum_{i=0}^l p_i n_i=p_0+\sum_{i=1}^l p_i n_i$. 

The labels $(p_1,\dots,p_l)$ determine the $\BZ$-grading $\g =\bigoplus_{j\in\BZ}\g(j)$ such that 
$\te\subset\g(0)$ and $\g^{\ap_i}\in \g(p_i)$ for $i=1,\dots,l$. Write $[\gamma:\ap_i]$ for the coefficient of
$\ap_i$ in the expression of $\gamma\in\Delta$ via $\Pi$. Letting 
$d(\gamma):=\sum_{i=1}^l [\gamma:\ap_i]p_i$,  we see that the root space $\g^\gamma$ belongs to 
$\g(d(\gamma))$. We say that $d(\gamma)$ is the $(\BZ,\vth)$-{\it degree}\/ of the root $\gamma$. For
this $\BZ$-grading, we have 
\begin{itemize}
\item \ $\p=\bigoplus_{j\ge 0}\g(j)=:\g({\ge}0)$ is a parabolic subalgebra of $\g$ with Levi subalgebra 
$\g(0)$,
\item \ $\n^-=\bigoplus_{j< 0}\g(j)=:\g({<}0)$ is the nilradical of an opposite parabolic subalgebra,
\end{itemize}
and $\g=\p\oplus\n^-$. In this setting, one has $d(\beta)\le d(\delta)$ for any $\beta\in\Delta^+$ and 
\beq            \label{eq:vysota}
   \max\{j\mid \g(j)\ne0\} =\sum_{i=1}^l n_i p_i =d(\delta)=m-p_0 \le m . 
\eeq
The $\BZ_m$-grading associated with $(p_0,p_1,\dots,p_l)$ is obtained from this $\BZ$-grading by 
``glueing''  modulo $m$. That is, for $j=0,1,\dots,m-1$, we set $\g_j=\bigoplus_{k\in\BZ}\g(j+km)$. The 
resulting decomposition 
\[
    \textstyle \g=\bigoplus_{j=0}^{m-1}\g_j
\]
is the $\BZ_m$-grading associated with $\vth=\vth(p_0,\dots,p_l)$. It follows from \eqref{eq:vysota}
that $\g_i=\g(i)\oplus \g(i-m)$ for $i=1,2,\dots,m-1$ (the sum of at most two spaces) and 
$\g_0=\g(-m)\oplus\g(0)\oplus\g(m)$ (at most three spaces). Moreover, $\g(0)=\g_0$ if and only if
$d(\delta)<m$, i.e., $p_0\ne 0$.

For $\mu\in \Delta$, let $\ov{d(\mu)}$ be the unique element of $\{0,1,\dots,m-1\}$ such that
$\g^\mu\subset \g_{\ov{d(\mu)}}$. Then
\begin{gather}   \label{eq:d-bar}
\text{if $1\le d(\mu)< m$, then $\ov{d(\mu)}=d(\mu)$ and $\ov{d(-\mu)}=m-d(\mu)$;} 
\\  
\text{  if $d(\mu)=0,\pm m$, then $\ov{d(\pm\mu)}=0$.}  \notag
\end{gather}
Using this description, we prove below that, for a wide class of inner automorphisms $\vth$, the 
$\vth$-contraction $\g_{(0)}$ admits a useful alternate description as a semi-direct product. Recall the
necessary setup. If $\h\subset\g$ is a subalgebra, then $\h\ltimes (\g/\h)^{\sf ab}$ stands for the 
corresponding {\it In\"on\"u--Wigner contraction}\/ of $\g$, see~\cite[Sect.\,2]{bn-oy}. Here the superscript 
``{\sf ab}''  means that the $\h$-module $\g/\h$ is an abelian ideal of this semi-direct product. 
Let
$\h=\p$ 
be a standard parabolic subalgebra associated with $\Pi$. 
Then $\g/\p$ can be identified with $\n^-$ as a vector space, and In\"on\"u--Wigner contractions of 
the form $\p\ltimes (\n^-)^{\sf ab}$, which have been studied in~\cite{sel}, are called {\it parabolic 
contractions} of $\g$.

\begin{thm}          \label{thm:semidir}
Suppose that $\vth\in {\sf Int}^f(\g)$ and $p_0=p_0(\vth)>0$. Let $\p$ and $\n^-$ be the subalgebras 
associated with $p_1,\dots,p_l$ as above. Then $\g_{(0)}\simeq \p\ltimes (\n^-)^{\sf ab}$.
\end{thm}
\begin{proof}
Since $p_0>0$, we have $\g(0)=\g_0$ and $d(\mu)<m$ for any $\mu\in\Delta^+$. Hence
$\ov{d(\mu)}=d(\mu)$ for {\bf every} $\mu\in\Delta^+$ and $\ov{d(-\mu)}=m-d(\mu)$ if $d(\mu)\ge 1$.
Set $\Delta(\p)=\{\gamma\in \Delta\mid d(\gamma)\ge 0\}$ and $\Delta(\n^-)=\Delta\setminus
\Delta(\p)$. Then $\Delta(\p)$ (resp. $\Delta(\n^-)$) is the set of roots of $\p$ (resp. $\n^-$).

Using this notation and the above relationship between $\BZ$ and $\BZ_m$-gradings, we now routinely 
verify that the Lie bracket in $\g_{(0)}$ coincides with that in $\p\ltimes (\n^-)^{\sf ab}$.

{\sf (1)} \ {\it The structure of\/ $(\p, [\ ,\,]_{(0)})$.}
If $\mu,\mu'\in\Delta(\p)$ and $\mu+\mu'$ is a root, then 
\[
    d(\mu), d(\mu'),d(\mu+\mu')\in [0, m-1].
\] 
(It is important here that $p_0>0$.) 
Then using \eqref{eq:g-and-g_0}, we get $[\g^{\mu},\g^{\mu'}]_{(0)}=[\g^{\mu},\g^{\mu'}]$. It is also clear that
$[\te,\g^{\mu}]_{(0)}=[\te,\g^{\mu}]$ for any $\mu\in\Delta(\p)$. Therefore, the
Lie brackets $[\ ,\,]$ and $[\ ,\,]_{(0)}$ coincide under the restriction to $\p$.

{\sf (2)} \ {\it The structure of\/ $(\n^-, [\ ,\,]_{(0)})$.} Let $d(\mu),d(\mu')\ge 1$, i.e., $-\mu,-\mu'\in \Delta(\n^-)$. Suppose that $\mu+\mu'$ is a root. Then
\[
    \ov{d(-\mu)}+\ov{d(-\mu')}=m-d(\mu)+(m-d(\mu'))=2m-d(\mu+\mu')>m .
\]
It follows that $[\g^{-\mu},\g^{-\mu'}]_{(0)}=0$, i.e., the space $\n^-$ is an abelian subalgebra of 
$\g_{(0)}$.

{\sf (3)} \ {\it The multiplication $[\p,\n^-]_{(0)}$.} Suppose that $\mu\in\Delta(\p)$, $-\mu'\in\Delta(\n^-)$,
and $\mu-\mu'\in\Delta$. 
\begin{itemize}
\item If $d(\mu')>d(\mu)$, then $\mu-\mu'\in\Delta(\n^-)$ and 
$\ov{d(\mu)}+\ov{d(-\mu')}=d(\mu)+m-d(\mu')<m$. Hence $[\g^{\mu},\g^{-\mu'}]_{(0)}=
[\g^{\mu},\g^{-\mu'}]\subset \n^-$.
\item If $d(\mu')\le d(\mu)$, then $\mu-\mu'\in\Delta(\p)$ and $\ov{d(\mu)}+\ov{d(-\mu')}\ge m$. Hence 
$[\g^{\mu},\g^{-\mu'}]_{(0)}=0$.
\item  It is also clear that $[\te,\g^{-\mu'}]_{(0)}=[\te,\g^{-\mu'}]$.
\end{itemize}
Thus, for all $x\in\p$ and $y\in\n^-$, the Lie bracket $[x,y]_{(0)}$ is computed as the initial 
bracket $[x,y]$ with the subsequent projection to $\n^-$ (w.r.{t.}{} the decomposition $\g=\p\oplus\n^-$).  
This precisely means that $\g_{(0)}$ and the semi-direct product $\p\ltimes(\n^-)^{\sf ab}$ are isomorphic 
as Lie algebras.
\end{proof}
 
Comparing our previous results for parabolic contractions $\p\ltimes(\n^-)^{\sf ab}$ (see~\cite{sel}) and 
$\BZ_m$-contractions $\g_{(0)}$ (see~\cite{OY,bn-oy,py21}), we gain new knowledge in both settings.

\begin{prop}       \label{prop:parab&index}
If $\vth\in {\sf Int}^f(\g)$ and $p_i(\vth)>0$ for some $i$ such that $n_i=1$, then $\g_{(0)}$ is a 
parabolic contraction of $\g$ and $\ind\g_{(0)}=\rk\g$.
\end{prop}
\begin{proof}
If $p_i(\vth)>0$ and $n_i=1$, then using an automorphism of $\tilde{\eus D}(\g)$, i.e., making another 
choice of $\be$, we can reduce the problem to the case $i=0$, see Section~\ref{subsub:int}. Hence 
$\g_{(0)}$ is a parabolic contraction by Theorem~\ref{thm:semidir}. 
By~\cite[Theorem\,4.1]{sel}, the index does not change for the parabolic contractions of $\g$,
i.e., $\ind(\p\ltimes(\n^-)^{\sf ab})=\ind \g$ for any parabolic subalgebra $\p\subset\g$.
\end{proof}

\begin{rmk}
If $p_i=0$ for all $i$ such that $n_i=1$, then 
the preceding approach fails and there seems to be no useful alternate description of $\g_{(0)}$.
\end{rmk}

The parabolic contractions of $\g$ are much more interesting than arbitrary In\"on\"u--Wigner 
contractions. Their structure is closely related to properties of the centralisers for the corresponding
Richardson orbit. Since $\p$ admits a complementary subspace $\n^-$, which is a Lie subalgebra,  the Lie--Poisson bracket 
associated with $\p\ltimes (\n^-)^{\sf ab}$ is compatible with the initial bracket on $\g$ 
(\cite[Lemma\,1.2]{bn-oy}). Then the Lenard--Magri scheme provides a {\sf PC} 
subalgebra of $\gS(\g)$, which is denoted by $\gZ(\p,\n^-)$. Let $[\ ,\,]_{(\p,\n^-)}$ denote the Lie bracket
for $\p\ltimes (\n^-)^{\sf ab}$. Then
we have the following properties of Poisson brackets and {\sf PC} subalgebras:
\begin{itemize}
\item[--] \ the {\sf PC}-subalgebra $\gZ(\g,\vth)$ is obtained via the application of the Lenard--Magri scheme to
the compatible Lie--Poisson brackets $[\ ,\,]$ and  $[\ ,\,]_{(0)}$;
\item[--] \  the {\sf PC}-subalgebra $\gZ(\p,\n^-)$ is obtained via the application of the Lenard--Magri scheme to
the compatible Lie--Poisson brackets $[\ ,\,]$ and $[\ ,\,]_{(\p,\n^-)}$;
\item[--] \  by Proposition~\ref{prop:parab&index}, if $p_i>0$ for some $i$ with $n_i=1$, then
$[\ ,\,]_{(0)}=[\ ,\,]_{(\p,\n^-)}$.
\end{itemize}
This leads to the following

\begin{cl}    \label{cor:Algebren=}
If $\vth\in {\sf Int}^f(\g)$ and $p_i>0$ for some $i$ such that $n_i=1$, then  $\gZ(\g,\vth)=\gZ(\p,\n^-)$.
\end{cl}

\begin{ex}   \label{ex:p=b}
Consider $\vth\in {\sf Int}^f(\g)$ such that $\g_0=\g(0)=\te$. This is equivalent to that $p_i>0$ for all
$i=0,1,\dots ,l$. Then $\p=\be$ is a Borel subalgebra and hence $\gZ(\g,\vth)=\gZ(\be,\ut^-)$. The 
advantage of this situation is that $\ut^-=[\be^-,\be^-]$ is a spherical subalgebra, and our results for the
{\sf PC} subalgebra $\gZ(\be,\ut^-)$ are more precise and complete~\cite[Sect.\,4,\,5]{bn-oy}. 
Namely, 
\begin{itemize}
\item[\sf (i)] \  $\trdeg \gZ(\be,\ut^-)=\bb(\g)$, the maximal possible value for the {\sf PC} subalgebras of
$\gS(\g)$;
\item[\sf (ii)] \ $\gZ(\be,\ut^-)$ is a maximal {\sf PC} subalgebra of $\gS(\g)$;
\item[\sf (iii)] \  $\gZ(\be,\ut^-)$ is a polynomial algebra, whose free generators are explicitly described.
\end{itemize}
Thus, results on parabolic contractions provide a description of $\gZ(\g,\vth)$ for a class of 
$\vth\in {\sf Int}^f(\g)$. 
(And it is not clear how to establish {\sf (ii)} and {\sf (iii)} in the context of $\BZ_m$-gradings!)
\end{ex}

Conversely, results on periodic contractions allow us to enrich the theory of parabolic contractions
and give a formula for $\trdeg\,\gZ(\p ,\n^-)$ with arbitrary $\p$.

\begin{prop}       \label{cl:podalgebra}
For any parabolic subalgebra $\p\subset\g$ with Levi subalgebra $\el$, we have
\[
     \trdeg\gZ(\p ,\n^-)=\bb(\g)-\bb(\el)+ \rk\g .
\]
\end{prop}
\begin{proof} Without loss of generality, we may assume that $\p\supset\be$ and $\el\supset\te$. Let 
$J\subset \{1,\dots,l\}$ correspond to the simple roots of $[\el,\el]$, i.e., $\ap_j\in\Pi$ is a root of $(\el,\te)$
if and only if $j\in J$. Take any $\vth\in {\sf Int}^f(\g)$ with the Kac labels $(p_0,\dots,p_l)$ such that 
$p_j=0$ if and only if $j\in J$ (in particular, $p_0\ne 0$). Then the $\BZ$-grading corresponding to
$(p_1,\dots,p_l)$ has the property that $\p=\g({\ge}0)$, $\el=\g(0)=\g_0$, and $\n^-=\g({<}0)$. Hence
$\g_{(0)}\simeq \p\ltimes (\n^-)^{\sf ab}$. On the other hand, since $\ind\g_{(0)}=\rk\g$ 
(Proposition~\ref{prop:parab&index}), we have
$\trdeg \gZ(\g,\vth)=\bb(\g)-\bb(\g_0)+\rk\g$, see~\cite[Theorem\,3.10]{py21}.
\end{proof}

Given $\vth$ with Kac labels $p_0,p_1,\dots,p_l$, the subalgebra $\g_0=\g^\vth$ depends only on the 
set $\caL(\vth):=\{i\in [0,l] \mid  p_i \ne 0\}$, see Section~\ref{subs:g_0}. (This also follows from the description of $\vth$-grading given above.) Let us prove that the similar property holds for the whole
$\vth$-contraction $\g_{(0)}$. That is, having replaced all {\bf nonzero} Kac labels $p_i$ with $1$, one 
obtains another automorphism $\tilde\vth$ (of a smaller order), but the corresponding periodic 
contractions appear to be isomorphic. Note that it is {\bf not} assumed now that $p_0>0$.

\begin{thm}    \label{thm:change-to-1}
For any $\vth\in{\sf Int}^f(\g)$, the $\vth$-contraction $\g_{(0)}$ depends only on $\caL(\vth)\subset \{0,1,\dots,l\}$. 
\end{thm}
\begin{proof}
Recall that $m=|\vth|=\sum_{i=0}^l p_i n_i=\sum_{i\in \caL(\vth)}p_i n_i$. Let $\tilde\vth$ denote the 
periodic automorphism such that $\caL(\vth)=\caL(\tilde\vth)$ and the nonzero Kac labels of $\tilde\vth$ 
are equal to $1$. Then $\tilde m:=|\tilde\vth|=\sum_{i\in \caL(\vth)}n_i$ and, for any $\beta\in\Delta$, its 
$(\BZ,\tilde\vth)$-degree equals $\tilde d(\beta):=\sum_{i\in\caL(\vth)}[\beta:\ap_i]$.
Write $\tilde\g_{(0)}$ for the $\tilde\vth$-contraction of $\g$ and then $[\ ,\,]^{\sim}_{(0)}$ stands for 
the corresponding Lie bracket. Our goal is to prove that $[\ ,\,]_{(0)}=[\ ,\,]^{\sim}_{(0)}$.

{\sf (1)} \ Both $\g_{(0)}$ and $\tilde\g_{(0)}$ share the same subalgebra $\g_0$. For any 
$x\in\g_{0}$ and $y\in\g$, we have $[x,y]_{(0)}=[x,y]=[x,y]^\sim_{(0)}$. In particular, this is true if $x\in\te$.

{\sf (2)} \ By linearity, our task is reduced to comparing the Lie brackets for two root 
spaces. For any $\beta,\mu\in \Delta$, one has either $[\g^\beta,\g^\mu]_{(0)}=[\g^\beta,\g^\mu]$ or 
$[\g^\beta,\g^\mu]_{(0)}=0$. Therefore, we have to check that if $[\g^\beta,\g^\mu]\ne 0$, then the 
property that $[\g^\beta,\g^\mu]_{(0)}=0$ depends only on $\caL(\vth)$. In other words, it suffices to prove 
that \ $[\g^\beta,\g^\mu]_{(0)}=0 \Longleftrightarrow [\g^\beta,\g^\mu]^\sim_{(0)}=0$. By (1),
we may also assume that $\beta,\mu\not\in \Delta(\g_0)$, i.e., $\ov{d(\beta)}\ne 0$ and $\ov{d(\mu)}\ne 0$.

\textbullet \quad Let $\beta,\mu\in\Delta^+\setminus \Delta(\g_0)$. Then $\ov{d(\beta)}=d(\beta)$ and 
$\ov{d(\mu)}=d(\mu)$. Suppose that $\beta+\mu\in\Delta$, i.e.  $[\g^\beta,\g^\mu]\ne 0$. Then
\[
   \text{$[\g^\beta,\g^\mu]_{(0)}=0$ if and only if $d(\beta)+d(\mu)\ge m$.} 
\]
On the other hand, $d(\beta)+d(\mu)=d(\beta+\mu)\le m-p_0$, cf.~\eqref{eq:vysota}. 
Assuming that $[\g^\beta,\g^\mu]_{(0)}=0$, we obtain 
 $p_0=0$ and $d(\beta+\mu)=d(\delta)=m$. The latter implies that 
$[\beta:\ap_i]+[\mu:\ap_i]=n_i$ for each $i\in \caL(\vth)$. Hence 
$\tilde d(\beta+\mu)=\tilde d(\delta)=\tilde m$ as well and thereby $[\g^\beta,\g^\mu]^\sim_{(0)}=0$.

\textbullet \quad Let $\beta,\mu\in\Delta^-\setminus \Delta(\g_0)$. Then $\ov{d(\beta)}=m-d(-\beta)$ and 
$\ov{d(\mu)}=m-d(-\mu)$. Suppose that $\beta+\mu\in\Delta$, i.e.  $[\g^\beta,\g^\mu]\ne 0$. 
In this case, $\ov{d(\beta)}+\ov{d(\mu)}=2m-d(-\mu-\nu)\ge m$, i.e., $[\g^\beta,\g^\mu]_{(0)}=0$. The same 
conclusion is obtained for $[\ ,\,]^{\sim}_{(0)}$ as well.

\textbullet \quad Suppose that $\beta\in\Delta^+\setminus \Delta(\g_0)$, 
$\mu\in\Delta^-\setminus \Delta(\g_0)$, and $\beta+\mu\in\Delta$. Then 
$\ov{d(\beta)}+\ov{d(\mu)}=d(\beta)+m-d(-\mu)=m+d(\beta+\mu)$. Therefore, $[\g^\beta,\g^\mu]_{(0)}=0$ 
if and only if $m+d(\beta+\mu)\ge m$, i.e., $\beta+\mu\in \Delta^+\cup \Delta(\g_0)$. Thus, this condition 
refers only to $\Delta(\g_0)$, which is the same for $\vth$ and $\tilde\vth$.
\end{proof}

\begin{rmk}   \label{rem:special-case}
If $p_0\ne 0$, i.e., $0\in\caL(\vth)$, then $\g_{(0)}\simeq\p\ltimes(\n^-)^{\sf ab}$ (Theorem~\ref{thm:semidir}). It is also clear that $\p$ and $\n^-$ depend only on 
$\{j\in [1,l]\mid p_j\ne 0\}= \caL(\vth)\setminus\{0\}$. That is, in this special case Theorem~\ref{thm:change-to-1} readily follows from Theorem~\ref{thm:semidir}.
\end{rmk}

\begin{ex}    \label{ex:G2}
For the  Lie algebra $\g$ of type $\GR{G}{2}$, one has ${\sf Aut}(\g)={\sf Int}(\g)$. Let us prove that
$\ind\g_{(0)}=\ind\g\,({=}2)$ for any periodic automorphism $\vth$. Here $\delta=3\ap_1+2\ap_2$, hence
$n_1=3$ and $n_2=2$. 
The affine Dynkin diagram $\GRt{G}{2}$ is 

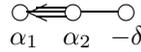
\centerline{ 
\begin{picture}(56,22)(0,-5)
\multiput(10,8)(20,0){3}{\circle{6}}
\put(6,-4){\footnotesize $\ap_1$} \put(26,-4){\footnotesize $\ap_2$}
\put(44,-4){\footnotesize $-\delta$}
\put(33,8){\line(1,0){14}}
\put(12.5,8){\line(1,0){14}}
\multiput(16.28,6.3)(0,3.4){2}{\line(1,0){11.22}}
\put(12.1,5){$<$} \end{picture}} 
 
\noindent
and the Kac diagram of $\vth=\vth(p_0,p_1,p_2)$ is
\begin{picture}(56,19)(0,5)
\multiput(10,8)(20,0){3}{\circle{6}}
\put(6.1,13.9){\footnotesize $p_1$} \put(26.1,13.9){\footnotesize $p_2$}
\put(44.4,13.9){\footnotesize $p_0$}
\put(33,8){\line(1,0){14}}
\put(12.5,8){\line(1,0){14}}
\multiput(16.28,6.3)(0,3.4){2}{\line(1,0){11.22}}
\put(12.1,5){$<$} \end{picture}, with $|\vth|=p_0+3p_1+2p_2$. 
By Proposition~\ref{prop:parab&index} and Theorem~\ref{thm:change-to-1}, it suffices to consider the cases, where $p_0=0$ and $(p_1,p_2)\in\{(0,1), (1,0), (1,1)\}$. Hence $|\vth|$ equals $2,3,5$, respectively.

Since $\ind\g_{(0)}=\ind\g$ for $|\vth|\le 3$ (Section~\ref{sect:index-g_0}), only the last case requires 
some consideration. The description of inner periodic automorphisms given above shows that here 
$\g_0=\te\oplus\g^{\delta}\oplus\g^{-\delta}$ and $\g_1$ is the sum of root spaces for
$\ap_1,\ap_2, -3\ap_1-\ap_2$. As $\g^{\ap_1}\oplus\g^{\ap_2}$ contains a regular nilpotent element of 
$\g$, see~\cite[Theorem\,4]{ko63}, so does $\g_1$  and hence $\ind\g_{(0)}=\ind\g$, 
cf.~\cite[Prop.\,5.3]{p09}.
\end{ex}

\begin{prop}    \label{prop:An}
If\/ $\g=\gt{sl}_{l+1}$ and $\vth\in {\sf Int}^f(\g)$, then $\g_{(0)}$ is a parabolic contraction of $\g$ and\/ $\ind\g_{(0)}=\ind\g=l$.
\end{prop}
\begin{proof}
For $\gt{sl}_{l+1}$, the affine Dynkin diagram $\GRt{A}{l}$ is a cycle and $n_i=1$ for all $i=0,1,\dots,l$. 
The Kac diagram of an inner automorphism is determined up to a rotation of this cycle. Therefore, we 
may always assume that $p_0>0$. Hence $\g_{(0)}$ is a parabolic contraction for {\bf every} 
$\vth\in {\sf Int}^f(\gt{sl}_{l+1})$  and thereby $\ind\g_{(0)}=\ind\g$ for {\bf all} inner periodic 
automorphisms.
\end{proof}

\begin{prop}    \label{prop:ind-sp-odd}
If\/ $\g=\gt{sp}_{2l}$ and $\vth\in {\sf Aut}^f(\g)$ with $|\vth|$ odd, then $\ind\g_{(0)}=\ind\g=l$.
\end{prop}
\begin{proof}
Here  ${\sf Aut}(\g)={\sf Int}(\g)$, $\delta=2\ap_1+\dots+2\ap_{l-1}+\ap_l$,  the affine Dynkin diagram $\GRt{C}{l}$ is  \\
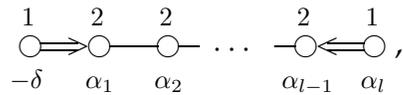
\centerline{
\begin{picture}(115,32)(75,-5)
\setlength{\unitlength}{0.018in} 
\multiput(70,8)(20,0){3}{\circle{6}}
\multiput(150,8)(20,0){2}{\circle{6}}
\put(93,8){\line(1,0){14}}
\multiput(156.34,7)(0,2){2}{\line(1,0){10.8}}    \put(152.78,5.72){$<$} 
\multiput(72.86,7)(0,2){2}{\line(1,0){10.8}}      \put(80.24,5.72){$>$}
\multiput(113,8)(28,0){2}{\line(1,0){6}}
\put(123,5){$\cdots$} 
\put(86,-4){\footnotesize $\ap_1$} \put(106,-4){\footnotesize $\ap_2$}
\put(143,-4){\footnotesize $\ap_{l{-}1}$} \put(166,-4){\footnotesize $\ap_l$}
\put(64,-4){\footnotesize $-\delta$}
\put(67.5,14.2){\footnotesize $1$}  \put(87.5,14.2){\footnotesize $2$} \put(107.5,14.2){\footnotesize $2$}
\put(147.5,14.2){\footnotesize $2$} \put(167.5,14.2){\footnotesize $1$}
\put(176,6){,}
\end{picture}
}
and the Kac diagram of $\vth=\vth(p_0,p_1,\dots,p_l)$ is \\
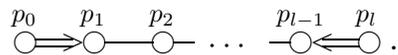
\centerline{
\begin{picture}(115,25)(75,3)
\setlength{\unitlength}{0.018in} 
\multiput(70,8)(20,0){3}{\circle{6}}
\multiput(150,8)(20,0){2}{\circle{6}}
\multiput(156.34,7)(0,2){2}{\line(1,0){10.8}}    \put(152.78,5.72){$<$} 
\multiput(72.86,7)(0,2){2}{\line(1,0){10.8}}      \put(80.24,5.72){$>$}
\put(93,8){\line(1,0){14}}
\multiput(113,8)(28,0){2}{\line(1,0){6}}
\put(123,5){$\cdots$}
\put(86,14.2){\footnotesize $p_1$} \put(106,14.2){\footnotesize $p_2$}
\put(143,14.2){\footnotesize $p_{l{-}1}$} \put(166,14.2){\footnotesize $p_l$}
\put(66,14.2){\footnotesize $p_0$}
\put(176,6){.}
\end{picture}
}
Here $|\vth|=p_0+2(p_1+\dots+ p_{l-1})+p_l$. By Theorem~\ref{thm:change-to-1}, we may assume that all $p_i\le 1$. Since $|\vth|$ is odd, either $p_0$ or $p_l$ is equal to $1$. Then Proposition~\ref{prop:parab&index}
applies.
\end{proof}

To provide yet another illustration of the interplay between parabolic contractions and $\vth$-contractions,
we need some preparations.

If $H\in\gS^d(\g)$, then one can decompose $H$ as the sum of bi-homogeneous components
$H=\sum_{i=0}^d  H_i$, where $H_i\in \gS^i(\n^-)\otimes \gS^{d-i}(\p)$. Then $H^\bullet_{\n^-}$ denotes 
the nonzero bi-homogeneous component of $H$ with maximal $i$ (=\, of maximal $\n^-$-degree). 

\begin{thm}[cf. Theorem\,5.1 in \cite{sel}]    
\label{thm:5.1-in-sel}
Let $\g$ be either $\gt{sl}_{l+1}$ or $\gt{sp}_{2l}$. If\/ $\q=\p\ltimes(\n^-)^{\sf ab}$ is any parabolic contraction of $\g$,
then $\gS(\q)^\q$ is a polynomial algebra. Moreover, there are free generators 
$H_1,\dots,H_l\in\gS(\g)^\g$ such that $(H_1)^\bullet_{\n^-}, \dots,(H_l)^\bullet_{\n^-}$ freely generate 
$\gS(\q)^\q$.
\end{thm}

In the situation of Theorem~\ref{thm:semidir}, we have $\g_{(0)}\simeq \p\ltimes(\n^-)^{\sf ab}$ and, for a
homogeneous $H\in\gS(\g)$, there are two {\sl a priori\/} different constructions:

\textbullet \quad First, one can take $H^\bullet$, the bi-homogeneous component of $H$ with highest 
$\vp$-degree. (Recall that this uses the $\BZ_m$-grading $\g=\bigoplus_{i=0}^{m-1}\g_i$ and 
$\vp:\bbk^*\to \GL(\g)$, see Section~\ref{subs:theta}.)

\textbullet \quad Alternatively, one can take $H^\bullet_{\n^-}$, which employs the direct sum 
$\g=\p\oplus\n^-$.
\\[.4ex]
However, the two decompositions of $\g$ are related in a very precise way, and therefore the following is not really surprising.

\begin{lm}              \label{lm:sovpad}
Suppose that $p_0(\vth)>0$, and let $\g=\bigoplus_{i=0}^{m-1}\g_i$ and $\g=\p\oplus\n^-$ be as above.
If $H\in\gS(\g)^\te$, then $H^\bullet=H^\bullet_{\n^-}$.
\end{lm}
\begin{proof}
Recall that if $p_0>0$, then $\g_0$ is a Levi subalgebra of $\p$, i.e., $\p=\g_0\oplus\n$. Take a basis for 
$\g$ that consists of the root vectors $e_\gamma$, $\gamma\in\Delta$, and a basis for $\te$.
Suppose that $H\in\gS(\g)^t$ is a monomial in that basis and $H\in \gS^i(\n^-)\otimes \gS^{\tilde j}(\p)$. Then
\[
     H=(\prod_{r=1}^i e_{-\gamma_r}){\cdot}f{\cdot}(\prod_{s=1}^j e_{\mu_s}) ,
\]
where $\gamma_1,\dots,\gamma_i\in \Delta(\n)$, $\mu_1,\dots,\mu_j\in \Delta(\p)$, 
$f\in\gS^{\tilde j-j}(\te)$, and $\gamma_1+\dots+\gamma_i=\mu_1+ \dots+\mu_j$. Let us compute
$\deg_\vp(H)$. By definition, $\deg_\vp(e_\gamma)=\ov{d(\gamma)}\in\{0,1,\dots,m-1\}$ and 
$\deg_\vp(f)=0$. For $\gamma\in\Delta(\n)$, we always have $\ov{d(-\gamma)}=m-{d(\gamma)}$;
and since $p_0>0$, we also have $\ov{d(\mu)}=d(\mu)$ for $\mu\in \Delta(\p)$, see \eqref{eq:d-bar}. Therefore,
\[
   \deg_\vp(H)=\sum_{r=1}^i (m-{d(\gamma_r)})+\sum_{s=1}^j {d(\mu_s)}=mi .
\]
Hence the $\vp$-degree of a $\te$-invariant monomial depends only on its $\n^-$-degree. Thus, if
$H\in \gS(\g)^\te$ is written in the basis above, then both $H^\bullet$ and $H^\bullet_{\n^-}$ consist of
the monomials of maximal $\n^-$-degree, and thereby $H^\bullet=H^\bullet_{\n^-}$. 
\end{proof}

The following is the promised ``illustration''.
\begin{thm}     \label{thm:svob-alg-sl}
For any $\vth\in {\sf Int}^f(\sln)$, there is a\/ {\sf g.g.s{.}} in $\gS(\sln)^{\sln}$ and the\/ {\sf PC} subalgebra 
$\gZ(\sln,\vth)$ is polynomial.
\end{thm}
\begin{proof}
We assume below that $n=l+1$.
By Theorem~\ref{thm:5.1-in-sel}, there is a set $H_1,\dots,H_l$ of free homogeneous generators of 
$\gS(\g)^\g$ such that $(H_1)^\bullet_{\n^-},\dots,(H_l)^\bullet_{\n^-}$ freely generate $\gS(\q)^\q$.
Under the hypothesis on $\vth$, we also have $\p\ltimes(\n^-)^{\sf ab}\simeq\g_{(0)}$ 
(Theorem~\ref{thm:semidir}) and 
$H_i^\bullet=(H_i)^\bullet_{\n^-}$ for each $i$ (Lemma~\ref{lm:sovpad}). This means that
\[
   \cz_0=\gS(\g_{(0)})^{\g_{(0)}}=\bbk[H_1^\bullet,\dots, H_l^\bullet]
\]
is a polynomial algebra and $H_1,\dots,H_l$ is a {\ggs} with respect to $\vth$.
By Theorem~\ref{thm:2.3}, we conclude that $\cz_0\subset\gZ_{\times}$ and that $\gZ_{\times}$ is  a polynomial algebra. 

\textbullet \ If $\g_0$ is not abelian, then $\infty\in \BP_{\sf sing}$ and hence $\gZ_\times=\gZ(\sln,\vth)$ is a polynomial algebra.  

\textbullet \ If $\g_0$ is abelian, then $\g_0=\te$, $\p=\be$, and $\g_{(0)}\simeq \be\ltimes(\ut^-)^{\sf ab}$. 
In this case, $\infty\in \BP_{\sf reg}$ and one has also to include $\cz_{\infty}$ in $\gZ(\sln,\vth)$. 
However, it was directly proved in \cite[Theorem\,4.3]{bn-oy} that here $\gZ(\be,\ut^-)=\gZ(\sln,\vth)$ is 
a polynomial algebra.
\end{proof}

\section{Modification of Kac diagrams for the outer automorphisms}
\label{sect:outer}

\noindent
Here we prove an analogue of Theorem~\ref{thm:change-to-1} to the {\bf outer} periodic automorphisms 
of simple Lie algebras. Let $\vth\in{\sf Aut}^f(\g)$ be outer, with the associated diagram automorphism 
$\sigma$, see Section~\ref{subs:kac}. Recall that $r=\rk\g^\sigma$ and 
$\Pi^{(\sigma)}=\{\nu_1,\dots,\nu_r\}$ is the set of simple roots of $\g^\sigma$. 

Let $\boldsymbol{p}=(p_0,p_1,\ldots,p_{r})$ be the Kac labels of $\vth$. Using $\boldsymbol{p}$,
we construct below the vector space sum 
$\g=\bigoplus_{j\in\BZ}\g(j)$. Unlike the case of inner automorphisms, this decomposition is {\bf not} 
going to be a Lie algebra grading on the whole of $\g$. Nevertheless, it will be compatible with the 
$\sigma$-grading \eqref{eq:sigma-grad}, and it will provide a Lie algebra $\BZ$-grading on $\g^\sigma$. 

{\bf --}  \ The $\BZ$-grading of $\g^\sigma$ is given by the conditions:
\begin{itemize}
\item \ $\te^\sigma\subset\g^\sigma(0)\subset\g(0)$;
\item \  for each $\nu_i\in\Pi^{(\sigma)}$, the root space $(\g^\sigma)^{\nu_i}$ belongs to 
$\g^\sigma(p_i)\subset\g(p_i)$. 
\end{itemize}

{\bf --}  \ 
For the lowest weight $-\delta_1$ of $\g^{(\sigma)}_1$, we set
$(\g^{(\sigma)}_1)^{-\delta_1}\subset\g(p_0)$. Hence if $\gamma=
-\delta_1+\sum_{i=1}^r c_i\nu_i$ is an arbitrary weight of $\g^{(\sigma)}_1$, then
$(\g^{(\sigma)}_1)^{\gamma}\subset\g(p_0+\sum_{i=1}^r c_ip_i)$. This defines a structure of a
$\BZ$-graded $\g^\sigma$-module on $\g^{(\sigma)}_1$ and completes the construction, if $\Ord=2$.

{\bf --}  \ If $\Ord=3$, then $[\g^{(\sigma)}_1,\g^{(\sigma)}_1]=\g^{(\sigma)}_2$ and the $\BZ$-grading on 
the latter is uniquely determined by the condition that 
$[\g^{(\sigma)}_1(i),\g^{(\sigma)}_1(j)]=\g^{(\sigma)}_2(i+j)$.

For each $\g^{(\sigma)}_i$, the vector space sum obtained is compatible with the weight 
decomposition with respect to $\te^\sigma$. That is, for a $\te^\sigma$-weight space 
$(\g^{(\sigma)}_i)^\gamma\subset\g^{(\sigma)}_i$, one can point out the integer $j$ such that 
$(\g^{(\sigma)}_i)^\gamma \subset \g(j)$. Then we write $d_i(\gamma)$ for this $j$.
The preceding exposition shows that

\ $d_0(\gamma)=\sum_{i=1}^r [\gamma:\nu_i]{\cdot}p_i$;    

\ $d_1(\gamma)=p_0+\sum_{i=1}^r [(\gamma+\delta_1):\nu_i]{\cdot} p_i$;

\ $d_2(\gamma)=2p_0+\sum_{i=1}^r [(\gamma+2\delta_1):\nu_i]{\cdot} p_i$.
\\
%
We say that $d_i(\gamma)$ is the $(\BZ,\vth)$-{\it degree}\/ of the weight $\gamma$ of $\g^{(\sigma)}_i$.  
The $\BZ_m$-grading of $\g$ associated with $\vth=\vth(\boldsymbol{p})$ is obtained from the graded 
vector space decomposition of $\g$ by ``glueing''  modulo 
$m=\Ord{\cdot}(p_0+\sum_{i=1}^{r} [\delta_1:\nu_i]{\cdot}p_i)=\Ord{\cdot}d_1(0)$. 

\begin{lm}              \label{in-eq}
For an outer $\vth\in {\sf Aut}(\g)$ with Kac labels $(p_0,p_1,\dots,p_r)$, we have 
\begin{itemize}
\item[\sf (i)] \ $0\le d_0(\beta)\le m$ for all $\beta\in\Delta^+(\g^\sigma)$;
\item[\sf (ii)] \ $jp_0 \le d_j(\gamma)\le m$ for any $\te^\sigma$-weight $\gamma$ of $\g_j^{(\sigma)}$, 
$j=1,2$. Moreover, the upper bound $m$ is attained if and only if $p_0=0$.
\end{itemize}
\end{lm}
\begin{proof}
{\sf (i)} Since $d_0(\nu_i)=p_i\ge 0$ for $i=1,\dots,r$, we obtain $d_0(\beta)\ge 0$ for any 
$\beta\in\Delta^+(\g^\sigma)$. It then suffices to check the inequality $d_0(\beta)\le m$ only for 
$\beta=\delta^\sigma$, the highest root in $\Delta^+(\g^\sigma)$.
We do this case-by-case. 

\textbullet \quad Suppose that $\Ord=2$. Let us compare the expressions of $\delta^\sigma$ and 
$\delta_1$ via $\Pi^{(\sigma)}$. Recall that $a'_i=[\delta_1:\nu_i]$. Set $a_i=[\delta^\sigma:\nu_i]$, 
$\boldsymbol{a}=(a_1,\dots,a_r)$, and $\boldsymbol{a'}=(a_1',\ldots,a_{r}')$. Then we have
\begin{itemize}
\item[]  for $\GR{A}{2n+1}$,  \ $\boldsymbol{a}=(2,2,\ldots,2,1)$ and $\boldsymbol{a'}=(1,2,\ldots,2,1)$;
\item[]  for $\GR{A}{2n}$, \ $\boldsymbol{a}=(1,2,\ldots,2,2)$ and $\boldsymbol{a'}=(2,2,\ldots,2,2)$.   
\item[]  for $\GR{D}{n}$, \ $\boldsymbol{a}=(1,2,\ldots,2)$ and $\boldsymbol{a'}=(1,1,\ldots,1)$; 
\item[]  for $\GR{E}{6}$, \ $\boldsymbol{a}=(2,4,3,2)$ and $\boldsymbol{a'}=(2,3,2,1)$. 
\end{itemize}
In all cases, $a_i\le \Ord{\cdot}a'_i=2a_i'$ for all $i$, whence the assertion.

\textbullet \quad If $\Ord=3$, then $\g=\gt{so}_8$ and $\g^\sigma$ is of type $\GR{G}{2}$.
Here $\delta^\sigma=3\nu_1+2\nu_2$ and $\delta_1=2\nu_1+\nu_2$ is the first fundamental weight of 
$\GR{G}{2}$. Then $d_0(\delta^\sigma)=3p_1+2p_2$ and $m=3(p_0+2p_1+p_2)$.  Hence 
$d_0(\delta^\sigma)\le m$.

{\sf (ii)} \ For the weights of $\g^{(\sigma)}_1$, the $(\BZ,\vth)$-degrees range from $d_1(-\delta_1)=p_0$, 
the degree of the lowest weight, until $d_1(\delta_1)=p_0+2\sum_{i=1}^r a_i'p_i$, the degree of the 
highest weight. Since $\Ord\ge 2$, we have then $m\ge 2(p_0+\sum_{i=1}^r a_i'p_i)$ and the result 
follows.

In case  $\Ord=3$,  the $(\BZ,\vth)$-degrees for the weights of $\g^{(\sigma)}_2$ range from 
$d_2(-\delta_1)=2p_0+a_1'p_1+ a_2' p_2$  until $d_2(\delta_1)=2p_0+3(a'_1p_1+a'_2 p_2)$.
And now $m=3(p_0+a'_1p_1+a'_2 p_2)$.

In any case, $d_{\Ord-1}(\delta_1)=m$ if and only if $p_0=0$.
\end{proof}

We set $\caL(\vth):=\{i\mid 0\le i\le r, \, p_i\ne 0\}$. If $x\in\g(j)\cap \g^{(\sigma)}_i$, then we also set 
$d(x)=j$.  For an integer $d$, let $\ov{d}$ be the unique element of $\{0,1,\dots,m-1\}$ such that 
$d-{\ov{d}}\in m\BZ$.

\begin{thm}            \label{thm:change2}
If $\vth\in {\sf Aut}(\g)$ is outer, then the Lie algebra $\g_{(0)}$ depends only on the set $\caL(\vth)$. 
\end{thm}
\begin{proof}
With necessary alterations, we follow the proof of Theorem~\ref{thm:change-to-1}.
The Lie algebra $\g_0$ depends only on $\caL(\vth)$.  If $x\in\g_0$ and $y\in\g$, then $[x,y]_{(0)}=[x,y]$.
We always assume below that $x,y\not\in\g_0$. Furthermore, $x$ and $y$ are weight vectors of $\te^\sigma$ in all cases. 

{\sf 1.} \quad 
We have either $[x,y]_{(0)}=[x,y]$ or 
$[x,y]_{(0)}=0$, see~\eqref{eq:g-and-g_0}. Therefore, one has to check that if $[x,y]\ne 0$, then the 
property that $[x,y]_{(0)}=0$ depends only on $\caL(\vth)$. 

If $[x,y]\in\g_0$, then $[x,y]_{(0)}=0$, since $x,y\not\in\g_0$.
For given $x$ and $y$, the condition $[x,y]\in\g_0$ depends only on $\caL(\vth)$.
Therefore we may safely assume that $[x,y]\not\in\g_0$, in particular, that $[x,y]\ne 0$.

From~\eqref{eq:g-and-g_0} one readily deduces the following
\begin{equation} \label{d-incr}
[x,y]_{(0)}=0 \ \text{ if and only if } \ \ov{d([x,y])}<\ov{d(x)} \ \text{ and/or } \ \ov{d([x,y])}<\ov{d(y)}. 
\end{equation}
 
{\sf 2.}  \quad
Suppose first that $x\in(\g^{\sigma})^\mu$, where $\mu\in\Delta^+(\g^\sigma)$. 
Using Lemma~\ref{in-eq} and the assumption $[x,y]\not\in\g_0$, we obtain 
\[
\ov{d([x,y])}=d([x,y])=d(x)+d(y)=\ov{d(x)}+\ov{d(y)}, 
\]
if $y\in\gt u^{\sigma}$ or $y\in\gt m$. 
Now by~\eqref{d-incr}, we have $[x,y]_{(0)}\ne 0$ in those cases.  
\\[.4ex]
 (\textbullet)  \enskip  
It remains to consider the case, where  
$y\in(\g^{\sigma})^\beta$ with 
$\beta\in\Delta^-(\g^\sigma)$. 
Here 
$[x,y]_{(0)}=0$ if and only if
\[
     d_0(\mu)+m-d_0(\beta)\ge m,
\]
which is equivalent to $d_0(\mu-\beta)\ge 0$. The last inequality holds if and only if 
$[x,y]\in\gt n^\sigma+\g_0$. For given $x$ and $y$, it depends only on $\caL(\vth)$.   

{\sf 3.} \quad Suppose next that $x\in(\g^{\sigma})^\mu$,  $x\in(\g^{\sigma})^\beta$ with $\mu,\beta\in\Delta^-(\g^\sigma)$. 
Here we have 
\[
 \ov{d(x)}+\ov{d(y)}=    m-d_0(-\mu)+m-d_0(-\beta)=2m-d_0(-\mu-\beta)\ge m, 
\]
where the inequality holds by Lemma~\ref{in-eq}{\sf (i)}. Hence $[x,y]_{(0)}$ in this case.

{\sf 4.} \quad 
Suppose  that  $x\in(\g^{\sigma})^\mu$ with $\mu\in\Delta^-(\g^\sigma)$, while 
 $y\in\gt m^\gamma$ is a weight vector of $\te_0$ and an eigenvector of $\sigma$. 
 Here we have 
 \[
\ov{d([x,y])}=d([x,y])=d(y)-d_0(-\mu)<d(y)=\ov{d(y)}
\]
and $[x,y]_{(0)}=0$ by~\eqref{d-incr}. 

{\sf 5.} \quad Now we consider the case, where both $x,y\in\gt m$ are weight vectors of $\te_0$ and  
eigenvectors of $\sigma$. Set $\be_{j}^{(\sigma)}=\be\cap \g_j^{(\sigma)}$. 
\\[.4ex]
 (\textbullet)  \enskip
Assume first that $\Ord=2$. Then $\gt m=\g_1^{(\sigma)}$ and $[\gt m,\gt m]\subset\g^\sigma$. 
By the construction, $\te_1^{(\sigma)}=\te\cap \g_1^{(\sigma)}\subset \g(m/2)$. 

If $x,y\in\be_{1}^{(\sigma)}$, then the $(\BZ,\vth)$-degree of $x$, as well as of $y$, is larger than or 
equal to $m/2$, but smaller than $m$ by Lemma~\ref{in-eq}{\sf (ii)}. Hence $[x,y]_{(0)}=0$. If 
$x,y\in\gt u^{-}\cap\g_1^{(\sigma)}$,   then $d(x)\le m/2$ and $d(y)\le m/2$. 
Here we have $[x,y]_{(0)}=[x,y]$, since $[x,y]\not\in\g_0$.

Suppose that $x\in\be_{1}^{(\sigma)}$ and $y\in\gt u^{-}\cap\g_1^{(\sigma)}$. Write $x\in \gt m^\mu$, 
$y\in\gt m^\beta$, where $\mu,\beta$ are weights of $\te^\sigma$, then 
$\mu+\beta\in\Delta(\g^\sigma)$, sinse $[x,y]\not\in\g_0$.   Note that $\gt m^{-\beta}\ne 0$, since $\gt m$ is a 
self-dual $\g^\sigma$-module. This applies to every $\te^\sigma$-weight in $\gt m$.  

Suppose that $\mu+\beta=\gamma\in\Delta^+(\g^\sigma)$. 
 Then $\mu=-\beta+\gamma$ and $d_1(\mu)=d_0(\gamma)+d_1(-\beta)$ with $d_1(-\beta)=m-d_1(\beta)$, cf.~Lemma~\ref{in-eq}. 
 Now 
\[
\ov{d(x)}+\ov{d(y)}=d(x)+d(y)=d_1(\mu)+d_1(\beta)=d_0(\gamma)+m-d_1(\beta)+d_1(\beta)=m+d_0(\gamma)\ge m
\]
and therefore $[x,y]_{(0)}=0$.

Suppose now that $\mu+\beta=-\gamma\in\Delta^-(\g^\sigma)$. 
Then, analogously, 
\[
      \ov{d(x)}+\ov{d(y)}=  d(x)+d(y)=d_1(\mu)+d_1(\beta)=d_1(-\beta)-d_0(\gamma)+d_1(\beta)=m-d_0(\gamma)\le m. 
\] 
Since $[x,y]\not\in\g_0$, 
 the inequality is strict and $[x,y]_{(0)}=[x,y]\ne 0$. \\[.4ex]
 (\textbullet)  \enskip
The case of $\Ord=3$ is similar.  Recall that $[\g_1^{(\sigma)},\g_1^{(\sigma)}]=\g_2^{(\sigma)}$, 
$[\g_1^{(\sigma)},\g_2^{(\sigma)}]=\g^{\sigma}$, and $[\g_2^{(\sigma)},\g_2^{(\sigma)}]=\g_1^{(\sigma)}$.
The $(\BZ,\vth)$-degrees of elements of $\g^{(\sigma)}_1$ range from $p_0$ to $p_0+2(2p_1+p_2)$. 
The maximal sum $d(x)+d(y)$ with $x,y\in\g_1^{(\sigma)}$ such that $[x,y]\ne 0$ is 
$m-p_0\le m$. Thereby here $[x,y]_{(0)}\ne0$, since $[x,y]\not\in\g_0$. 

The minimal sum $d(x)+d(y)$ with $x,y\in\g_2^{(\sigma)}$ such that $[x,y]\ne 0$ is 
$m+p_0\ge m$. Thereby here  $[x,y]_{(0)}=0$ for all elements. 

 Suppose that $x\in\g_{1}^{(\sigma)}$ and $y\in \g_{2}^{(\sigma)}$. 
Write $x\in(\g^{(\sigma)}_1)^\mu$, $y\in(\g^{(\sigma)}_2)^\beta$, where $\mu,\beta$ are $\te^\sigma$-
weights. Then $\mu+\beta\in\Delta(\g^\sigma)$, sinse $[x,y]\not\in\g_0$.  

Suppose that $\mu+\beta=\gamma\in\Delta^+(\g^\sigma)$.
Then 
\[
      \ov{d(x)}+\ov{d(y)}=   d(x)+d(y)=m+d_0(\gamma)\ge m
\]
and therefore $[x,y]_{(0)}=0$.

Finally suppose  that $\alpha+\beta=-\gamma\in\Delta^-(\g^\sigma)$. Then 
\[
       \ov{d(x)}+\ov{d(y)}=   d(x)+d(y)=m-d_0(\gamma)\le m. 
\] 
Since $[x,y]\not\in\g_0$, the inequality is strict and $[x,y]_{(0)}=[x,y]\ne 0$.
\end{proof}

\section{The index of periodic contractions of the orthogonal Lie algebras}
\label{sect:son}

In this section, we prove that $\ind\g_{(0)}=\ind\g$ for any $\vth\in{\sf Aut}^f(\g)$, if $\g=\soN$. To this
end, we need Vinberg's description of the periodic automorphisms for the classical Lie algebras and related Cartan subspaces in $\g_1$~\cite[\S\,7]{vi76}.

In the rest of the section, we work with $\g=\soN=\soVB$, where 
$\VV=\bbk^N$ 
and $\caB$ is a symmetric non-degenerate bilinear form on $\VV$.

If $\vth\in{\sf Aut}(\soN)$ and $|\vth|=m$, then $\vth=\vth_A$ is the conjugation
with a matrix $A\in\Op(\VV, \caB)$ such that $A^m=\pm I_N$. Set $\VV(\lb)=\{v\in\VV\mid Av=\lb v\}$. 
Then $\VV=\bigoplus_{\lb\in S}\VV(\lb)$, where either $S=\{\lb\mid \lb^m=1\}$ or $S=\{\lb\mid \lb^m=- 1\}$. Clearly, 
$\caB(\VV(\lb),\VV(\mu))=0$ unless $\lb\mu=1$.  Hence $\dim \VV(\lb)=\dim\VV(\lb^{-1})$.

Suppose that $A^m=I_N$. 
Then 
$S=\{1,\zeta,\dots,\zeta^{m-1}\}$, and we set $b_j=\dim\VV(\zeta^j)$ for $j=0,1,\dots,m-1$. Note that
$b_j=b_{m-j}$ for $j\ge 1$. 

If $\vth_A$ is outer, then $N=2l$ is even, $m$\/ is also even, and $\det(A)=-1$. The latter implies
that $\dim\VV(-1)$ is odd, hence $\VV(-1)\ne 0$. 
 We see  that  $A^{m}=I_N$.
 Since $\dim\VV(-1)$ is odd and $\dim\VV$ is even, 
$b_0=\dim\VV(1)$ is also odd and hence $b_0\ne 0$ as well as $b_{m/2}=\dim\VV(-1)$.

\begin{lm}                 \label{lm:so-ou}
Let $\vth$ be an outer periodic automorphism of $\g=\gt{so}_{2l}$ such that the Kac labels of $\vth$ are 
zeros and ones. Then $\g_1$ contains a nonzero semisimple element.
\end{lm}
\begin{proof}
We have $\vth=\vth_A$ with $A\in\Op_{2l}$ and $\det(A)=-1$; as above, $A^m=I_N$. In~\cite[\S\,7.2]{vi76}, Vinberg gives a formula for $\rk(\g_0,\g_1)$ (i.e., the dimension of a Cartan subspace in $\g_1$) in terms of the $A$-eigenspaces in $\VV$. In the present setting, we have the so-called automorphism of type~I, and then 
$\rk(\g_0,\g_1)=\min \{b_0,b_1,\dots,b_{m/2}\}$. We already know that $b_0,b_{m/2}\ge 1$. 

The spectrum of $A$ in $\VV$ shows that the centraliser of $A$ in $\gt{so}_{2l}\simeq\wedge^2\VV$ is 
\[
     \g_0=\gt{so}_{b_0}\oplus\gt{gl}_{b_1}\oplus\ldots\oplus\gt{gl}_{b_{(m/2)-1}}\oplus\gt{so}_{b_{m/2}}.
\]
On the other hand, we can use the Kac diagram $\eus K(\vth)$ and the hypothesis that the labels does 
not exceed $1$. Here $\g^\sigma=\gt{so}_{2l-1}$, $r=l-1$, and
the twisted affine Dynkin diagram $\GR{D}{l}^{(2)}$ equipped with the coefficients $(a'_0,a'_1,\dots,a'_{l-1})$ over the nodes is

\centerline{
\begin{picture}(120,38)(75,-7)
\setlength{\unitlength}{0.018in} 
\multiput(70,8)(20,0){3}{\circle{6}}
\multiput(150,8)(20,0){2}{\circle{6}}
\multiput(76.34,7)(0,2){2}{\line(1,0){10.8}}    \put(72.78,5.72){$<$} 
\multiput(152.86,7)(0,2){2}{\line(1,0){10.8}}      \put(160.24,5.72){$>$}
\put(93,8){\line(1,0){14}}
\multiput(113,8)(27,0){2}{\line(1,0){7}}
\put(124,5){$\cdots$}
\put(67.5,14.2){\footnotesize $1$}  \put(87.5,14.2){\footnotesize $1$} \put(107.5,14.2){\footnotesize $1$}
\put(147.5,14.2){\footnotesize $1$} \put(167.5,14.2){\footnotesize $1$}
\put(63,-4){\footnotesize $-\delta_1$}
\put(86.5,-4){\footnotesize $\nu_1$} \put(106.5,-4){\footnotesize $\nu_2$}
\put(143.5,-4){\footnotesize $\nu_{l{-}2}$} \put(166.5,-4){\footnotesize $\nu_{l-1}$}
\put(176,6){.}
\end{picture} }

\noindent
Since $m=|\vth|=\Ord\bigl(\sum_{i=0}^{l-1}p_i(\vth)a'_i\bigr)=2\bigl(\sum_{i=0}^{l-1}p_i(\vth)\bigr)$ is even 
and $p_i(\vth)\le 1$, the Kac diagram contains $m/2$ nonzero labels. This implies that 
$\eus K(\vth)$ is of the following form:

$\eus K(\vth)$: \quad 
\raisebox{-4.5ex}{\begin{picture}(300,50)(0,-20)
\setlength{\unitlength}{0.018in} 
\multiput(0,8)(20,0){2}{\circle{6}}     \multiput(60,8)(20,0){2}{\circle{6}}
\multiput(60,8)(20,0){2}{\circle{6}}   \multiput(120,8)(40,0){2}{\circle{6}}
\multiput(200,8)(20,0){2}{\circle{6}} \multiput(260,8)(20,0){2}{\circle{6}}
\multiput(93,5)(40,0){3}{$\cdots$}  \multiput(33,5)(200,0){2}{$\cdots$}
\multiput(63,8)(140,0){2}{\line(1,0){14}}
\multiput(23,8)(28,0){2}{\line(1,0){6}} \multiput(83,8)(28,0){2}{\line(1,0){6}}
\multiput(123,8)(28,0){2}{\line(1,0){6}} \multiput(163,8)(28,0){2}{\line(1,0){6}}
\multiput(223,8)(28,0){2}{\line(1,0){6}}
\multiput(6.34,7)(0,2){2}{\line(1,0){10.8}}    \put(2.78,5.72){$<$} 
\multiput(262.86,7)(0,2){2}{\line(1,0){10.8}}      \put(270.24,5.72){$>$}
\multiput(57.5,14.2)(60,0){2}{\footnotesize {\it\bfseries 1}}
\multiput(157.5,14.2)(60,0){2}{\footnotesize {\it\bfseries 1}}
\put(-4,4){$\underbrace{\mbox{\hspace{57\unitlength}}}_{b' \ \text{nodes}}$}
\put(72,4){$\underbrace{\mbox{\hspace{41\unitlength}}}_{s_1 \ \text{nodes}}$}
\put(133,-9){$\cdots$}   \put(133,13){$\cdots$}
\put(168,4){$\underbrace{\mbox{\hspace{41\unitlength}}}_{s_k \ \text{nodes}}$}
\put(227,4){$\underbrace{\mbox{\hspace{57\unitlength}}}_{b'' \ \text{nodes}}$} 
\put(287,6){,}
\end{picture}  }

\noindent 
where the zero Kac labels are omitted and $k=(m/2)-1$. According to the description of $\g_0$ via the 
Kac diagram (Section~\ref{subs:g_0}), we obtain here 
\[
  \g_0=\gt{so}_{2b'+1}\oplus \bigl(\bigoplus_{i=1}^{(m/2)-1}\gt{gl}_{s_i+1}\bigr)\oplus \gt{so}_{2b''+1} .
\]
Hence $\{b_0, b_{m/2}\}=\{2b'+1,2b''+1\}$ and 
$\{b_1,\dots,b_{(m/2)-1}\}=\{s_1+1,\dots,s_{(m/2)-1}+1\}$. Thus, $b_j\ge 1$ for all $j$ and hence
$\rk\!(\g_0,\g_1)\ge 1$, i.e., $\g_1$ contains nonzero semisimple elements.
\end{proof}

\begin{lm}            \label{lm:so-in}
Let $\vth$ be an inner periodic automorphism of\/ $\g=\soN$ such that $p_i(\vth)\in \{0,1\}$ 
for all $i$.  Furthermore, assume that $p_i(\vth)=0$ for all $i$ such that $n_i=1$, i.e., \\
\centerline{
$\begin{array}{ll}
p_0(\vth)=p_1(\vth)=p_{l-1}(\vth)=p_l(\vth)=0, & \text{ if\/ $\g$ is of type $\GR{D}{l}$}, \\ 
p_0(\vth)=p_1(\vth)=0, & \text{ if\/ $\g$ is of type $\GR{B}{l}$}. 
\end{array}$ }
\\
Then $\g_1$ contains a nonzero semisimple element.
\end{lm}
\begin{proof}
Since $\vth$ is inner, we may assume that $\vth=\vth_A$, where $A\in SO(\VV,\caB)$, i.e., 
$\det A=1$. We have \  
$(n_0,n_1,\ldots,n_{l-1},n_l)= \begin{cases}
(1,1,2,\ldots,2,1,1) & \text{ in type  $\GR{D}{l}$}, \\ 
(1,1,2,\ldots,2) & \text{ in type $\GR{B}{l}$}. 
\end{cases}$ 
\\
Therefore the assumptions on the Kac labels imply that $m$ is even and exactly $m/2$ labels are equal 
to $1$.

If $\g$ is of type $\GR{D}{l}$, then the Kac diagram of $\vth$ has $l+1$ nodes and looks as follows:

$\eus K(\theta)$: \quad 
\raisebox{-6.3ex}{\begin{picture}(300,64)(0,-30)
\setlength{\unitlength}{0.018in} 
\multiput(0,-4)(0,24){2}{\circle{6}}   \put(20,8){\circle{6}} 
\put(3,-2){\line(5,3){13.5}}  \put(3,18){\line(5,-3){13.5}} 
\multiput(60,8)(20,0){2}{\circle{6}}
\multiput(60,8)(20,0){2}{\circle{6}}   \multiput(120,8)(40,0){2}{\circle{6}}
\multiput(200,8)(20,0){2}{\circle{6}} \put(260,8){\circle{6}}
\multiput(280,-4)(0,24){2}{\circle{6}}
\put(277,-2){\line(-5,3){13.5}}  \put(277,18){\line(-5,-3){13.5}} 
\multiput(93,5)(40,0){3}{$\cdots$}  \multiput(33,5)(200,0){2}{$\cdots$}
\multiput(63,8)(140,0){2}{\line(1,0){14}}
\multiput(23,8)(28,0){2}{\line(1,0){6}} \multiput(83,8)(28,0){2}{\line(1,0){6}}
\multiput(123,8)(28,0){2}{\line(1,0){6}} \multiput(163,8)(28,0){2}{\line(1,0){6}}
\multiput(223,8)(28,0){2}{\line(1,0){6}}
\multiput(57.5,14.2)(60,0){2}{\footnotesize {\it\bfseries 1}}
\multiput(157.5,14.2)(60,0){2}{\footnotesize {\it\bfseries 1}}
\put(-4,-8){$\underbrace{\mbox{\hspace{56\unitlength}}}_{b' \ \text{nodes}}$}
\put(72,4){$\underbrace{\mbox{\hspace{41\unitlength}}}_{s_1 \ \text{nodes}}$}
\put(133,-9){$\cdots$}  \put(133,13){$\cdots$}
\put(168,4){$\underbrace{\mbox{\hspace{41\unitlength}}}_{s_k \ \text{nodes}}$}
\put(228,-8){$\underbrace{\mbox{\hspace{56\unitlength}}}_{b'' \ \text{nodes}}$}
\put(287,6){,}
\end{picture} }

\noindent
where $k=(m/2)-1$. By the assumption on Kac labels, we have $b',b''\ge 2$. Hence
$\g_0$ has the non-trivial summands $\gt{so}_{2b'}$, $\gt{so}_{2b''}$ and $(m/2)-1$  nonzero summands 
$\gt{gl}_{s_i+1}$.   If $A^m=-I_{2l}$, then neither $1$ nor $-1$ is 
an eigenvalues of $A$, since $m$ is even.  
Hence 
the centraliser of $A$ in $\gt{so}_{2l}$, i.e., $\g_0$, 
is a sum of  $m/2$ summands $\gt{gl}_{\dim\VV(\lambda)}$ with $\lambda^m=-1$. 
It has fewer summands that required by $\eus K(\vth)$. Therefore $A^m=I_{2l}$ and the eigenvalues of $A$ are $m$-th roots of unity.  Arguing as in the proof of Lemma~\ref{lm:so-ou}, we obtain that each 
$m$-th root of unity is an eigenvalue of $A$. In this case, the automorphism $\vth$ is again of type I in 
the sense of Vinberg~\cite[\S\,7.2]{vi76} and hence 
$\rk\!(\g_0,\g_1)=\min\limits_{0\le j\le m/2}\{ b_j\}\ge 1$. Thus, $\g_1$ contains nonzero semisimple elements. 

If $\g$ is of type $\GR{B}{l}$, then the argument is similar. The difference is that $\dim\VV=2l+1$ and
the Kac diagram of $\vth$ (having $l+1$ nodes) looks as follows:

$\eus K(\vth)$: \quad 
\raisebox{-6.3ex}{\begin{picture}(300,64)(0,-30)
\setlength{\unitlength}{0.018in} 
\multiput(0,-4)(0,24){2}{\circle{6}}   \put(20,8){\circle{6}} 
\put(3,-2){\line(5,3){13.5}}  \put(3,18){\line(5,-3){13.5}} 
\multiput(60,8)(20,0){2}{\circle{6}}
\multiput(60,8)(20,0){2}{\circle{6}}   \multiput(120,8)(40,0){2}{\circle{6}}
\multiput(200,8)(20,0){2}{\circle{6}} \multiput(260,8)(20,0){2}{\circle{6}}
\multiput(93,5)(40,0){3}{$\cdots$}  \multiput(33,5)(200,0){2}{$\cdots$}
\multiput(63,8)(140,0){2}{\line(1,0){14}}
\multiput(23,8)(28,0){2}{\line(1,0){6}} \multiput(83,8)(28,0){2}{\line(1,0){6}}
\multiput(123,8)(28,0){2}{\line(1,0){6}} \multiput(163,8)(28,0){2}{\line(1,0){6}}
\multiput(223,8)(28,0){2}{\line(1,0){6}}
\multiput(262.86,7)(0,2){2}{\line(1,0){10.8}}      \put(270.24,5.72){$>$}
\multiput(57.5,14.2)(60,0){2}{\footnotesize {\it\bfseries 1}}
\multiput(157.5,14.2)(60,0){2}{\footnotesize {\it\bfseries 1}}
\put(-4,-8){$\underbrace{\mbox{\hspace{57\unitlength}}}_{b' \ \text{nodes}}$}
\put(72,4){$\underbrace{\mbox{\hspace{41\unitlength}}}_{s_1 \ \text{nodes}}$}
\put(133,-9){$\cdots$}    \put(133,13){$\cdots$}
\put(168,4){$\underbrace{\mbox{\hspace{41\unitlength}}}_{s_k \ \text{nodes}}$}
\put(227,4){$\underbrace{\mbox{\hspace{57\unitlength}}}_{b'' \ \text{nodes}}$}
\put(288,6){,}
\end{picture} }

\noindent
where $k=(m/2)-1$ and $b'\ge 2$. 
Since $\dim\VV$ is odd, $1$ or $-1$ has to be an eigenvalue of $A$. Therefore $A^m=I_{2l+1}$ and 
again we have $b_j\ge 1$ for all $0\le j\le m/2$.
\end{proof}

\begin{thm}       \label{thm:so-ind}
If\/ $\g=\soN$, then $\ind\g_{(0)}=\rk\g$ \ for any periodic automorphism $\vth$.
\end{thm}
\begin{proof}
We argue by induction on $N+m$ with $m=|\vth|$. If $m\le 3$, then the statement holds by 
Proposition~\ref{prop:2} and \cite{p07}.  Clearly, it holds also for $N\le 3$, cf.~Proposition~\ref{prop:An}. 

If there is a Kac label of $\vth$ that is larger than $1$, then we may replace it with `1' without changing 
the Lie algebra structure of $\g_{(0)}$, see Theorems~\ref{thm:change-to-1} and \ref{thm:change2}.   
Clearly, $m$ decreases under this procedure. Therefore we may assume that the Kac labels of $\vth$ 
belong to $\{0,1\}$. 

If $\vth$ is inner and at least one of the labels $p_0, p_1, p_{l-1}, p_l$ in type $\GR{D}{l}$ equals `1' or 
one of the labels $p_0, p_1$ in type $\GR{B}{l}$ equals `1', then $\ind\g_{(0)}=\rk\g$ by  
Proposition~\ref{prop:parab&index}. 

Therefore, we may assume that either $\vth$ is outer or $\vth$ is inner with 
$p_0=p_1=p_{l-1}=p_l=0$ (in  type $\GR{D}{l}$) and $p_0=p_1=0$ (in  type $\GR{B}{l}$). 
This implies that $m$ is even and $\g_1$ contains a nonzero semisimple element $x$, see 
Lemmas~\ref{lm:so-ou} and~\ref{lm:so-in}. By Corollary~\ref{cor:2}, it suffices to prove that 
$\ind(\g^x)_{(0)}=\ind\g^x$ for some $x\in\g_1$. Let $x=C_i\in\g_1$ be
one of the basis semisimple elements defined in~\cite[\S\,7.2]{vi76}. 
As an endomorphism of $\VV$, it has the following properties: 
\begin{itemize}
\item[($\diamond$)] \ $x{\cdot} \VV(\lambda)$ is a $1$-dimensional subspace of 
$\VV(\zeta\lambda)$ for each $\lambda\in S$; 
\item[($\diamond$)]  \ $x^m\ne 0$.
\end{itemize}
These properties imply  
that 
$\g^x=\gt{so}_{N-m} \oplus\te_{m/2}$, where $\te_{m/2}$ is an abelian Lie algebra of dimension $m/2$. 
Since $[\g^x,\g^x]$ is a smaller orthogonal Lie algebra, the induction hypothesis applies, which completes the proof. 
\end{proof}

\begin{rmk}
For $\g=\gt{sp}_{2l}$, we have ${\sf Aut}(\g)={\sf Int}(\g)$, but
an analogue of Lemma~\ref{lm:so-in} is not true. Here 
$(n_0,n_1,\ldots,n_{l-1},n_l)=(1,2,\ldots,2,1)$ and it may happen that $p_0(\vth)=p_l(\vth)=0$, but
$\g_1$ contains no nonzero semisimple elements, i.e., $\g_1\subset\N$.
In this case, $m$ is necessarily even. The simplest example of such $\vth$ occurs if $p_i=p_{i+1}=1$ for certain $i$ with $1\le i\le l-2$ and all other $p_j$ are zero, see the Kac diagram below:

\centerline{ 
$\eus K(\vth)$: \quad 
\raisebox{-4ex}{\begin{picture}(160,46)(0,-20)
\setlength{\unitlength}{0.018in} 
\multiput(0,8)(20,0){2}{\circle{6}} \multiput(60,8)(20,0){2}{\circle{6}}
\multiput(120,8)(20,0){2}{\circle{6}}
\multiput(126.34,7)(0,2){2}{\line(1,0){10.8}}    \put(122.78,5.72){$<$} 
\multiput(2.86,7)(0,2){2}{\line(1,0){10.8}}      \put(10.24,5.72){$>$}
\put(63,8){\line(1,0){14}}
\multiput(23,8)(28,0){2}{\line(1,0){6}}   
\multiput(83,8)(28,0){2}{\line(1,0){6}}
\put(33,5){$\cdots$}  \put(93,5){$\cdots$}
\put(57.5,14.2){\footnotesize {\it\bfseries 1}} \put(77.5,14.2){\footnotesize {\it\bfseries 1}} 
\put(-4,1){$\underbrace{\mbox{\hspace{56\unitlength}}}_{i \ \text{nodes}}$}
\put(88,1){$\underbrace{\mbox{\hspace{56\unitlength}}}_{l-i-1 \ \text{nodes}}$}
\put(150,6){,}
\end{picture}} 
}

\noindent 
Then $m=4$, $\g_1\subset\N$, and $\ind\g_{(0)}$ is not known. Here $\g_0=\gt{sp}_{2i}\oplus
\gt{sp}_{2j}\oplus\te_1$, where $j=l-i-1$.
\end{rmk}

\section{$\gN$-regular automorphisms and good generating systems}
\label{sect:ggs-&-N-reg}

\noindent
In this section, we prove that if $\vth$ is an $\gN$-regular automorphism of $\g$, then
$\vth$ admits a good generating system and obtain some related results on the structure of the 
{\sf PC} subalgebras $\gZ_\times, \gZ(\g,\vth) \subset \gS(\g)^{\g_0}$. Moreover, if $\tilde\vth$ is ``close'' 
to an $\gN$-regular automorphism (see Def.~\ref{def:friendly}), then $\tilde\vth$ also admits a g.g.s.  

As before, we assume that $\vth\in{\sf Aut}^f(\g)$, $|\vth|=m$, and $\zeta=\sqrt[m]1$ is a primitive root
of unity. Let $H_1,\dots, H_l$ be a set of $\vth$-{generators} in $\gS(\g)^\g$ and $\deg H_j=d_j$. We
have $\vth(H_j)=\esi_j H_j$ and $\esi_j=\zeta^{r_j}$ for a unique $r_j\in\{0,1,\dots,m-1\}$.

Following~\cite[Sect.\,3]{p05}, we associate to $\vth$ the set of integers $\{k_i\}_{i=0}^{m-1}$ defined as follows:
\[
   k_i=\#\{j\in [1,l]\mid \zeta^{m_j}\esi_j=\zeta^i\}=\#\{j\in [1,l]\mid m_j+r_j\equiv i\!\!\pmod m\}.
\]
Then $\sum_i k_i=l$. The eigenvalues $\{\esi_j\}$ depend only on the image of $\vth$ in
${\sf Aut}(\g)/{\sf Int}(\g)$ (denoted $\bar\vth$), i.e., on the connected component of ${\sf Aut}(\g)$ that 
contains $\vth$. Therefore, the vector $\vec{k}=\vec{k}(m,\bar\vth)=(k_0,\dots,k_{m-1})$ depends only on 
$m$ and $\bar\vth$.
We say that the tuple $(|\vth|, \vec{k})$ is the {\it datum} of a periodic automorphism $\vth$.

If $F\in\bbk[\g]^G$, then $F\vert_{\g_1}\in\bbk[\g_1]^{G_0}$. However, the restriction homomorphism
\\
\centerline{$\psi_1: \bbk[\g]^G\to \bbk[\g_1]^{G_0}$, \ $F\mapsto F\vert_{\g_1}$} 
\\
is not always onto. As a modest contribution to the invariant theory of $\vth$-groups, we record the 
following observation.

\begin{prop}       \label{prop:vklad}
Let $\vth$ be an arbitrary periodic automorphism of $\g$. Then
\begin{itemize}
\item[\sf (i)] \ $\bbk[\g_1]^{G_0}$ is integral over $\psi_1(\bbk[\g]^G)$;
\item[\sf (ii)] \   if the datum of $\vth$ is $(m,k_0,\dots,k_{m-1})$, then
$\trdeg\bbk[\g_1]^{G_0}=\dim\g_1\md G_0\le k_{m-1}$.
\end{itemize}
\end{prop}
\begin{proof}
{\sf (i)} \ By~\cite[\S\,2.3]{vi76}, $\N\cap\g_1=:\N_1$ is the null-cone for the $G_0$-action on $\g_1$. 
Therefore, the polynomials $H_1\vert_{\g_1}, \dots,H_l\vert_{\g_1}$ have the same zero locus as the
ideal in $\bbk[\g_1]$ generated by the augmentation ideal $\bbk[\g_1]^{G_0}_+$ in $\bbk[\g_1]^{G_0}$. 
By a result of Hilbert~(1893), this implies that $\bbk[\g_1]^{G_0}$ is integral over 
$\bbk[H_1\vert_{\g_1}, \dots,H_l\vert_{\g_1}]=\psi_1(\bbk[\g]^G)$. 
\\ \indent
(For a short modern proof of Hilbert's result, we refer to~\cite[Theorem\,2]{kempf}.)

{\sf (ii)} \ If $\deg H_j=d_j$ and $H_j(x)\ne 0$ for some $x\in \g_1$, then
\[
 \zeta^{d_j} H_j(x)=H_j(\zeta x)=H_j(\vth(x))=(\vth^{-1}H_j)(x)=\esi_j^{-1}H_j(x) .
\]
Hence $m_j+r_j\equiv m-1 \pmod m$. Therefore, there are at most $k_{m-1}$ \ $\vth$-generators $\{H_j\}$ that do not vanish on $\g_1$, and the assertion follows from {\sf (i)}.
\end{proof}

\begin{df}    \label{def:N-reg}
A periodic automorphism $\vth$ is said to be $\gN$-{\it regular}, 
if $\g_1$ contains a regular nilpotent 
element of $\g$. 
\end{df}
Basic results on the $\gN$-regular automorphisms are obtained in \cite[Section\,3]{p05}:

\begin{thm}            \label{thm:P05}
If\/ $\vth$ is\/ $\gN$-regular and $|\vth|=m$, then 
\begin{itemize}
\item[\sf (i)] \ $\psi_1(\bbk[\g]^{G})=\bbk[\g_1]^{G_0}$ and $\dim\g_1\md G_0=k_{m-1}$;
\item[\sf (ii)] \ the dimension of a generic stabiliser for the $G_0$-action on $\g_1$ equals $k_0$. 
\end{itemize}
In particular, $\dim\g_0-k_0=\dim\g_1-k_{m-1}=\max\dim_{x\in\g_1}G_0{\cdot}x$. 
\end{thm}

Hence the $\gN$-regular automorphism are distinguished by the properties that 
the restriction homomorphism $\psi_1$ is onto and $\dim\g_1\md G_0$ has the maximal possible value 
among the automorphisms of $\g$ with a given datum.

\begin{rmk}      \label{rem:N-reg-mnogo}
If a connected component of ${\sf Aut}(\g)$ contains elements of order $m$, then it
contains $\gN$-regular automorphisms of order $m$, see~\cite[Theorem\,3.2]{p05}. Moreover, all these 
$\gN$-regular automorphisms of order $m$ are $G$-conjugate~\cite[Theorem\,2.3]{p05}. In particular,
for each $m\in\BN$, there is a unique, up to conjugacy, inner $\gN$-regular automorphism of order $m$.
\end{rmk}

\begin{prop}[{\cite[Thm.\,3.3(iv) \& Corollary\,3.4]{p05}}]   \label{prop:dim-g_i}
If\/ $\vth$ is\/ $\gN$-regular and $|\vth|=m$, then 
\begin{gather}      \label{eq:D_th-N-reg}
    \dim\g_0=\frac{1}{m} \bigl(\dim\g+\sum_{i=0}^{m-1} (m-1-2i)k_i\bigr) \ \text{ and} \\
     \dim\g_{i+1}-\dim\g_i=k_{m-1-i}-k_i            \label{eq:raznitsa}
\end{gather}
for every $i\in\{0,1,\dots,m{-}1\}$. 
\end{prop}
Clearly, this yields formulae for $\dim\g_i$ with all $i$. 

\noindent
Recall that $D_\vth=\sum_{i=0}^{m-1}i \dim\g_i$. Since $\dim\g_i=\dim\g_{m-i}$ for $i=1,2,\dots,m-1$,
one readily verifies that 
\beq        \label{eq:D_th}
D_\vth=\frac{m}{2}(\dim\g-\dim\g_0) .
\eeq
\begin{lm}   \label{lm:D_th-N-reg}
In the $\gN$-regular case, we have
\[
  D_\vth=\frac{1}{2}\bigl((m-1)\dim\g+\sum_{i=0}^{m-1}(2i+1-m)k_i\bigr)=
  \frac{m}{2}\bigl((m-1)\dim\g_0+\sum_{i=0}^{m-1}(2i+1-m)k_i\bigr).
\]
\end{lm}
\begin{proof}
Substitute the expression for either $\dim\g_0$ or $\dim\g$ from~\eqref{eq:D_th-N-reg} into~\eqref{eq:D_th}.
\end{proof}

Our next goal is to obtain an upper bound on the $\vp$-degree of $H_j$ (Section~\ref{subs:theta}). We 
recall the necessary setup, with a more elaborate notation. 
Using the vector space decomposition $\g=\g_0\oplus\g_1\oplus\ldots\oplus\g_{m-1}$, we write $H_j$ as 
the sum of multi-homogeneous components:
\beq   \label{eq:multi-hom}
    H_j=\bigoplus_{\bia} (H_j)_{\bia} ,
\eeq
where $\bia=(i_0,i_1,\dots,i_{m-1})$, \ $i_0+i_1+\dots +i_{m-1}=d_j$, and
\[
  (H_j)_{\bia}\in \gS^{i_0}(\g_0)\otimes \gS^{i_1}(\g_1)\otimes\dots \otimes\gS^{i_{m-1}}(\g_{m-1})
  \subset \gS^{d_j}(\g).
\]
Set $p(\bia)=i_1+2i_2+\dots+(m-1)i_{m-1}$. Then $\vp(t){\cdot}(H_j)_{\bia}=t^{p(\bia)}(H_j)_{\bia}$ and
$\vth((H_j)_{\bia})=\zeta^{p(\bia)}(H_j)_{\bia}$. Recall that $\vth(H_j)=\zeta^{r_j}H_j$. Hence
if $(H_j)_{\bia}\ne 0$, then  $p(\bia)-r_j\equiv 0 \pmod m$. Then 
\begin{itemize}
\item $d_j^\bullet:=\max \{p(\bia)\mid (H_j)_{\bia}\ne 0\}=\deg_\vp(H_j)$ is the $\vp$-{degree} of $H_j$; 
\item $H_j^\bullet$ \ is the sum of all multi-homogeneous components of $H_j$, where $p(\bia)$ is
maximal. 
\end{itemize}

\noindent
Whenever we wish to stress that $d^\bullet_j$ is determined via a certain $\vth$, we write 
$d^\bullet_j(\vth)$ for it. Recall that a set of $\vth$-generators $H_1,\dots, H_l$ is called a {\sf g.g.s.} 
with respect to $\vth$, if  $H_1^\bullet,\dots,H_l^\bullet$ are algebraically independent. 

\noindent
A $\vth$-generator $H_j$ is said to be {\it of type} {\sl (i)}, if $m_j+r_j\equiv i  \!\pmod m$ for
$i\in\{0,1,\dots,m-1\}$.

\begin{lm}    \label{lm:otsenka-d-bullet}
If $H_j$ is of type {\sl (i)}, then $d_j^\bullet\le (m-1)m_j+i$.
\end{lm}
\begin{proof}
By definition, $d_j^\bullet\le (m-1)d_j$ and $d_j^\bullet\equiv r_j \pmod m$. For the $m$-tuple
\[
      \bja= (\underbrace{0,\dots,0,1}_{i},0,\dots,0,m_j) ,
\]  
we have $p(\bja)=(m-1)m_j+i$ and  $p(\bja)-r_j=m m_j- (m_j+r_j-i) \equiv 0\pmod m$, i.e.,
$(H_j)_{\bja}$ may occur in $H_j$. Since
\[
     (m-1)m_j \le p(\bja)\le (m-1)d_j
\]
and $p(\bja)$ is the unique integer in this interval that is comparable with $r_i$ modulo $m$, we conclude
that $d_j^\bullet \le p(\bja)$.
\end{proof}

\begin{prop}    \label{prop:otsenka-summa}
For any $\vth\in {\sf Aut}^f(\g)$ with $|\vth|=m$, we have
\beq     \label{eq:otsenka-summa}
    \sum_{j=1}^l d_j^\bullet \le \frac{1}{2}\bigl((m-1)\dim\g+\sum_{i=0}^{m-1}(2i+1-m)k_i\bigr) .
\eeq
\end{prop}
\begin{proof}
Set $\eus P_i=\{j\in [1,l]\mid H_j \text{ is of type {\sl (i)}}\}$. Then 
$\#\eus P_i=k_i$ and $\bigcup_{i=0}^{m-1}\eus P_i=[1,l]$.
By Lemma~\ref{lm:otsenka-d-bullet}, we obtain
\[
   \sum_{j=1}^l d_j^\bullet \le \sum_{i=0}^{m-1}\left(\sum_{j\in\eus P_i}((m-1)m_j+i)\right)=
   (m-1)\sum_{j=1}^l m_j+ \sum_{i=0}^{m-1} i k_i .
\] 
Since $\sum_{j=1}^l m_j=\frac{1}{2}(\dim\g-l)$ and $l=\sum_i k_i$, the last expression is easily being
transformed into the RHS in~\eqref{eq:otsenka-summa}.
\end{proof}
Since $\vec{k}=(k_0,\dots,k_{m-1})$ depends only on $m$ and $\bar\vth$,
the upper bound in Proposition~\ref{prop:otsenka-summa} depends only on the datum of $\vth$.
Let $\mathfrak Y(m,\vec{k})$ denote this upper bound, i.e., the RHS in \eqref{eq:otsenka-summa}.

\begin{thm}      \label{thm:main1}
Suppose that $\vth\in {\sf Aut}^f(\g)$ is $\gN$-regular and $|\vth|=m$. Let $H_1,\dots,H_l$ be an 
arbitrary set of $\vth$-generators in $\gS(\g)^\g$. Then
\begin{itemize}
\item[\sf (1)] \ $d_j^\bullet=(m-1)m_j+i$ for any $H_j$ of type {\sl (i)};
\item[\sf (2)] \ $D_\vth=\sum_{j=1}^l  d_j^\bullet=\mathfrak Y(m,\vec{k})$;
\item[\sf (3)] \  $H_1,\dots,H_l$ is a\/ {\sf g.g.s.} with respect to $\vth$.
\end{itemize}
\end{thm}    
\begin{proof}
For any $\vth\in {\sf Aut}(\g)$, one has $D_\vth\le \sum_{j=1}^l  d_j^\bullet$, see~\cite[Theorem\,3.8]{contr} or Theorem~\ref{thm:kokosik}. 
On the other hand, for an $\gN$-regular $\vth$, combining Lemma~\ref{lm:D_th-N-reg}, 
Lemma~\ref{lm:otsenka-d-bullet}, and Proposition~\ref{prop:otsenka-summa} shows that 
$D_\vth\ge \sum_{j=1}^l  d_j^\bullet$. Therefore, there must be equalities in {\sf (2)} and also in {\sf (1)} 
for $j=1,\dots,l$.

Furthermore, a set of $\vth$-generators $H_1,\dots,H_l$ is a {\sf g.g.s.} with respect to $\vth$ if and only if $D_\vth=\sum_{j=1}^l  d_j^\bullet$, see again~\cite{contr}.
\end{proof}

{\bf Remark.} The point of {\sf (3)} is that if $\vth$ is $\gN$-regular, then {\bf any} set of $\vth$-generators is a {\sf g.g.s.} If $\vth$ is not $\gN$-regular, then it may happen that the property of being {\sf g.g.s.} depends on the choice of $\vth$-generators.

Decomposition~\eqref{eq:multi-hom} provides the bi-homogeneous decomposition 
$H_j=\bigoplus_i H_{j,i}$, where 
\[
   H_{j,i}:=\sum_{\bia:\ p(\bia)=i} (H_j)_\bia .
\]
Then $d_j^\bullet=\max\{i\mid H_{j,i}\ne0\}$ and if $H_{j,i}\ne0$, then $i\equiv r_j\pmod m$. These 
bi-homogeneous decompositions have already been studied in~\cite{py21}. In particular,
the subalgebra of $\gS(\g)$ generated by all bi-homogeneous components $\{H_{j,i}\}$ is {\sf PC} and it 
actually coincides with $\eus Z_\times$, see \cite[Eq. (4.1)]{py21}. 

\begin{thm}    \label{thm:main2}
Let $\vth$ be an $\gN$-regular automorphism of order $m$. Then 
\begin{itemize}
\item[\sf (i)]  \ all possible bi-homogeneous components of all $H_j$ are nonzero, i.e., $H_{j,i}\ne0$ if and
only if \ $0\le i\le d_j^\bullet$ and $i\equiv r_j\pmod m$;
\item[\sf (ii)] \  all these bi-homogeneous components are algebraically independent and therefore
$\eus Z_\times$ is a polynomial algebra;
\item[\sf (iii)] \ $\sum_{j=1}^l \left(\frac{d_j^\bullet-r_j}{m}+1\right)=\bb(\g,\vth)=\trdeg\eus Z_\times$.
\end{itemize}
\end{thm}
\begin{proof}
If $\vth$ is $\gN$-regular, then $\vth$ admits a {\sf g.g.s.}~(Theorem~\ref{thm:main1}) and the 
equality $\ind\g_{(0)}=\ind\g$ holds for the $\vth$-contraction of $\g$~\cite[Prop.\,5.3]{p09}. Therefore, all 
assertions directly follow from Theorems~4.3 and 4.6 in~\cite{py21}.
\end{proof}

There is a strong constraint on the Kac labels of $\gN$-regular inner automorphisms.
\begin{thm}    \label{thm:p_i=0,1}
Suppose that $\vth\in {\sf Int}^f(\g)$ is $\gN$-regular. Then 
\begin{itemize}
\item[\sf (i)] \ $p_i(\vth)\in\{0,1\}$ for all $i$ such that $n_i>1$; 
\item[\sf (ii)] \ if \ $p_i(\vth)>1$ for some $i$ such that $n_i=1$, then $p_j(\vth)=1$ for all other 
$j$. 
\end{itemize}
\end{thm}
\begin{proof}  
Let $\co_{\sf reg}$ be the $G$-orbit of regular nilpotent elements. By hypothesis, 
$\co_{\sf reg}\cap\g_1\ne \varnothing$.

{\sf (i)} \ 
Suppose that $p_j(\vth)>1$ for some $j$. Then $\g_1\subset\gN$~\cite[\S\,8.3]{vi76} (this also follows 
from the construction of the $\BZ_m$-grading in Section~\ref{sect:inner-auto}). The subdiagram of 
$\tilde{\eus D}(\g)$ without the $j$-th node gives rise to the regular semisimple subalgebra 
$\bar\g\subset\g$ with a set of simple roots $(\Pi\setminus\{\ap_j\})\cup\{-\delta\}$. Since $p_j(\vth)>1$, 
the induced $\BZ_m$-grading $\bar\g=\bigoplus_{i\in\BZ_m}\bar\g_i$ has the property that 
$\bar\g_1=\g_1$. Hence $\co_{\sf reg}\cap\bar\g\ne\varnothing$. 
On the other hand, $\bar\g$ is the fixed-point subalgebra of $\bar\vth\in {\sf Int}^f(\g)$, where $\bar\vth$
is defined by the Kac labels $p_j(\bar\vth)=1$ and $p_i(\bar\vth)=0$ for all other $i$. Hence 
$|\bar\vth|=n_j$. If $n_j>1$, then $\bar\vth$ is a non-trivial automorphism of $\g$ such that 
$\co_{\sf reg}\cap\g^{\bar\vth}\ne\varnothing$, which is impossible. Indeed, $\bar\vth={\sf Int}(x)$
for some non-central semisimple $x\in G$ and $x\in G^e$ for $e\in \co_{\sf reg}\cap\g^{\bar\vth}$. But
$G^e$ ($e\in\co_{\sf reg}$) contains no non-central semisimple elements. Thus, if $p_j(\vth)>1$, then 
$n_j=1$ and $\bar\g=\g$.

{\sf (ii)} \ Let $\Gamma$ denote the symmetry group of the affine Dynkin diagram $\tilde{\eus D}(\g)$. 
Since $\Gamma$ acts transitively on the set of nodes with $n_i=1$ and $\eus K(\vth)$ is determined up 
to the action of $\Gamma$, we may assume that $j=0$.  The remaining labels $p_1,\dots,p_m$ 
determine a $\BZ$-grading of $\g$ such that $\g(1)=\g_1$ and  $\co_{\sf reg}\cap\g(1)\ne\varnothing$.
Hence the corresponding nilradical $\n=\g({\ge}1)$ also meets $\co_{\sf reg}$. But this is only possible if
$\n=\ut=[\be,\be]$, i.e., $p_i\ge 1$ for $i=1,\dots,l$. Then 
$\g(1)=\bigoplus_{i\in \eus J}\g^{\ap_i}$, where $\eus J=\{i\in \{1,\dots,l\}\mid p_i=1\}$.
By \cite[Theorem\,4]{ko63}, this means that $\eus J=\{1,\dots,l\}$.
\end{proof}

Recall that the {\it Coxeter number\/} of $\g$ is $\mathsf h=\sum_{i=0}^l n_i=1+\sum_{i=1}^l[\delta:\ap_i]$.
\begin{cl}   \label{cor:0,1}
If $\vth$ is $\gN$-regular and $|\vth|\le \mathsf h$, then $p_i(\vth)\le 1$ for all $i$.
\end{cl}
Next result demonstrates another extreme property of $\gN$-regular automorphisms and its relationship with existence of {\sf g.g.s.}
\begin{thm}   \label{thm:main3}
Let $\vth$ and $\vth'$ have the same data (i.e., $|\vth|=|\vth'|$ and they belong to the same connected 
component of\/ ${\sf Aut}(\g)$). Suppose that $\vth$ is $\gN$-regular. Then
\begin{itemize}
\item[\sf (i)] \ $\dim\g^\vth\le \dim\g^{\vth'}$;
\item[\sf (ii)] \  if\/ $\dim\g^\vth= \dim\g^{\vth'}$, then
$\vth'$ also admits a {\sf g.g.s.} for {\it\bfseries any} set of $\vth'$-generators $H_1,\dots,H_l$.
\end{itemize}
\end{thm}
\begin{proof} Previous results of this section and \cite[Theorem\,3.8]{contr} imply that
\[
  D_{\vth'}\le \sum_{j=1}^l d_j^\bullet (\vth')\le \mathfrak Y(m,\vec{k}) =D_\vth .
\]
Since $D_\vth=\frac{m}{2}(\dim\g-\dim\g^\vth)$ for any $\vth$, we get {\sf (i)}. The above relation also 
implies that if $\dim\g^\vth= \dim\g^{\vth'}$, then  
$D_{\vth'}= \sum_{j=1}^j d_j^\bullet (\vth')=\mathfrak Y(m,\vec{k})$, and we can again refer to~\cite{contr}.
\end{proof}

\begin{rmk}
It can happen that $\sum_{j=1}^l d_j^\bullet (\vth')<\mathfrak Y(m,\vec{k})$, but still
$D_{\vth'}= \sum_{j=1}^l d_j^\bullet (\vth')$, i.e., $\vth'$ admits a {\sf g.g.s.}. If this happens to be the 
case, then not every set of $\vth'$-generators forms a {\sf g.g.s.}, and one has to make a right choice.
It is known that {\bf all} involutions of the classical Lie algebras admit a {\sf g.g.s.} regardless of 
$\gN$-regularity~\cite{contr}, and there are exactly four involutions for exceptional Lie algebras of type 
$\GR{E}{n}$  that do not admit a {\sf g.g.s.}~\cite{Y-imrn}.
\end{rmk}
The equality occurring in Theorem~\ref{thm:main3}{\sf (ii)} is not rare. 
Such non-conjugate pairs $(\vth,\vth')$ do exist for $m\ge 3$.
\begin{df}     \label{def:friendly}
We say that two non-conjugate automorphisms $\vth,\tilde\vth$ form a {\it friendly pair}, if they 
have the same data, $\vth$ is $\gN$-regular, and $\dim\g^\vth= \dim\g^{\tilde\vth}$.
\end{df}

Together with presence of {\sf g.g.s.}, the members of a friendly pair share other good properties.
To distinguish the $\BZ_m$-gradings for
$\vth$ and $\tilde\vth$, we write $\g=\bigoplus_{i=0}^{m-1}\g_i$ for $\vth$ (which is $\gN$-regular) 
and $\g=\bigoplus_{i=0}^{m-1}\tilde\g_i$ for $\tilde\vth$.

\begin{prop}
Let $(\vth,\tilde\vth)$ be a friendly pair. Then 
\begin{itemize}
\item[\sf (i)] \ $\dim\tilde\g_1\md {\tilde G}_0=\dim\g_1\md G_0=k_{m-1}$;
\item[\sf (ii)] \ if $\tilde H_1,\dots,\tilde H_l$ is any set of $\tilde\vth$-generators, then
$\{\tilde H_j\vert_{\tilde\g_1 }\mid j\in \eus P_{m-1}\}$ is a system of parameters 
in\/ $\bbk[\tilde\g_1]^{\tilde G_0}$.
\end{itemize}
\end{prop}
\begin{proof}
If $H_1,\dots,H_l$ is any set of $\vth$-generators, then the polynomials 
$\{H_j\vert_{\g_1}\mid j\in \eus P_{m-1}\}$ freely generate $\bbk[\g_1]^{G_0}$ 
(see~\cite[Theorem\,3.5]{p05} or Theorem~\ref{thm:P05}). Therefore, we only have to prove the assertions related to $\tilde\vth$. 

We assume below that $\tilde H_1,\dots,\tilde H_l$ is a set of $\tilde\vth$-generators.
It is shown in Proposition~\ref{prop:vklad} that if $j\not\in \eus P_{m-1}$, then 
$\tilde H_j\vert_{\tilde\g_1}=0$. On the other hand, since $\tilde H_1,\dots,\tilde H_l$ is a g.g.s. with 
respect to $\tilde\vth$, one has 
\[
   d_j^\bullet=(m-1)m_j+m-1=(m-1)d_j \ \text{ for } \ j\in\eus P_{m-1} .
\]
Therefore, $\tilde H_j^\bullet=(\tilde H_j)_\bia$ with $\bia=(0,\dots,0,d_j)$. Hence 
$\tilde H_j^\bullet\in\eus S^{d_j}(\g_{m-1})$, and the latter is the set of polynomial functions of degree 
$d_j$ on $\g_1\simeq (\g_{m-1})^*$. In other words, $\tilde H_j^\bullet$ is obtained as follows. We first 
take $\tilde H_j\vert_{\g_1}=\psi_1(\tilde H_j)$ and then consider it as function on the whole of $\g$ via 
the projection $\g\to \g_1$.

Because $\tilde H_1^\bullet,\dots,\tilde H_l^\bullet$ are algebraically independent in $\eus S(\g)$, we 
obtain that $\{\tilde H_j\vert_{\tilde\g_1 }\mid j\in \eus P_{m-1}\}$ are algebraically independent in 
$\eus S(\g_{m-1})=\bbk[\g_1]$. The rest follows from Proposition~\ref{prop:vklad}.
\end{proof}

\begin{rmk}   \label{rem:ne-znaem}
(1) For a friendly pair $(\vth,\tilde\vth)$, the polynomials 
$\{\tilde H_j\vert_{\tilde\g_1 }\mid j\in \eus P_{m-1}\}$ do not always generate 
$\bbk[\tilde\g_1]^{\tilde G_0}$. 

(2) Although $\tilde\vth$ admits a {\sf g.g.s.}~(Theorem~\ref{thm:main3}), we do not know in general whether the $\tilde\vth$-contraction of $\g$ has the same index as $\g$. 
\end{rmk}

\subsection{How to determine $\eus K(\vth)$ for $\gN$-regular inner automorphisms}
We provide some hints that are sufficient in most cases.

\textbullet\quad If $m\ge {\sf h}$, then $p_i(\vth)=1$ for $i=1,\dots,l$ and $p_0=m+1-{\sf h}$.

\textbullet\quad Suppose that $m <{\sf h}$.
\\
{\bf --} \ Since $p_i(\vth)\in\{0,1\}$ (Corollary~\ref{cor:0,1}), it suffices to determine the subset $J\subset
\{0,1,\dots,l\}$ such that $p_j=1$ if and only if $j\in J$. The obvious condition is that $\sum_{j\in J}n_j=m$.
If there are several possibilities for such $J$, then one can compare $\dim\g_0$ and $\dim\g_1$
obtained from these $J$ with those required by Proposition~\ref{prop:dim-g_i}.
\\
{\bf --} \ For any $m\in\BN$, there is an explicit construction of an $\gN$-regular inner $\vth$ with $|\vth|=m$. 
Let $\g=\bigoplus_{i\in\BZ}\g(i)$ be the standard $\BZ$-grading. This means that $\te\subset\g(0)$ and
$\g(1)=\bigoplus_{\ap\in\Pi}\g^\ap$. Then $\g^\gamma\subset \g(\hot(\gamma))$ for any 
$\gamma\in\Delta$, where $\hot(\gamma)=\sum_{\ap\in\Pi}[\gamma:\ap]$. Here $\co_{\sf reg}\cap\g(1)$
is dense in $\g(1)$. Hence glueing this $\BZ$-grading module $m$ yields the unique, up to $G$-conjugacy, 
$\gN$-regular $\vth$ of order $m$. For $m< {\sf h}$, this construction does not allow us to see the Kac 
labels of $\vth$. Nevertheless, one easily determines $\g_0$, because the root system of $[\g_0,\g_0]$ is
$\Delta^{(m)}=\{\gamma\in \Delta\mid \hot(\gamma)\in m\BZ\}$. This gives a strong constraint on possible 
subsets $J$.
\\
{\bf --} \ To realise that $\vth$ is not $\gN$-regular, one can use Theorem~\ref{thm:P05}(i). That is, if 
$\bbk[\g_1]^{G_0}$ has a free generator of degree that does not belong to 
$\{d_j \mid j\in \eus P_{m-1}\}$, then $\vth$ cannot be $\gN$-regular.

In our examples of friendly pairs, the Kac 
labels belong to $\{0,1\}$, and the zero labels are omitted. Let 
$\overrightarrow{\dim}(\vth)$ be the vector $(\dim\g_0,\dim\g_1,\dots,\dim\g_{m-1})$ 
for $\vth$ with $|\vth|=m$. The numbers $\dim\g_0$ and $\dim\g_1$ can directly be read off the Kac 
diagram, see Section~\ref{subs:g_0}. Since $\dim\g_i=\dim\g_{m-i}$ for $i\ne 0$, the knowledge of
$\dim\g_0$ and $\dim\g_1$ is sufficient for obtaining $\overrightarrow{\dim}(\vth)$, if $m\le 5$. 
The Lie algebra of an $n$-dimensional algebraic torus is denoted by $\te_n$.

\begin{ex}   \label{ex:friendly-pairs} 
{\bf 1}$^o$. For $\g$ of type $\GR{E}{7}$, we consider the following inner automorphisms:  

$\eus K(\vth)$ :\quad \raisebox{-3.8ex}
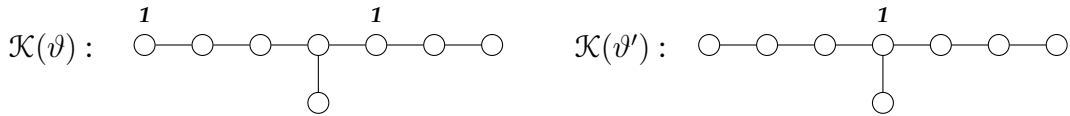
{\begin{tikzpicture}[scale= .70, transform shape]
\tikzstyle{every node}=[circle, draw, fill=white!55] 
\node (h) [label=above:{\it\bfseries 1}] at (-1.1,0) {};
\node (a) at (0,0) {};
\node (b)  at (1.1,0) {};
\node (c)  at (2.2,0) {};
\node (d) [label=above:{\it\bfseries 1}] at (3.3,0) {};
\node (e) at (4.4,0) {};
\node (f)  at (5.5,0) {};
\node (g) at (2.2,-1.1) {};
\tikzstyle{every node}=[circle, draw, fill=violet!55] 
\foreach \from/\to in {h/a, a/b, b/c, c/d, d/e, e/f, c/g}  \draw[-] (\from) -- (\to);
\end{tikzpicture}}  \qquad
$\eus K(\vth')$ :\quad \raisebox{-3.8ex}{\begin{tikzpicture}[scale= .70, transform shape]
\tikzstyle{every node}=[circle, draw, fill=white!55] 
\node (h) at (-1.1,0) {};
\node (a) at (0,0) {};
\node (b) at (1.1,0) {};
\node (c)  [label=above:{\it\bfseries 1}] at (2.2,0) {};
\node (d) at (3.3,0) {};
\node (e) at (4.4,0) {};
\node (f)  at (5.5,0) {};
\node (g) at (2.2,-1.1) {};
\tikzstyle{every node}=[circle, draw, fill=violet!55] 
\foreach \from/\to in {h/a, a/b, b/c, c/d, d/e, e/f, c/g}  \draw[-] (\from) -- (\to);
\end{tikzpicture}} 

\vspace{.8ex}
Then $\g^\vth=\GR{A}{4}\oplus\GR{A}{2}\oplus\te_1$, \ 
$\g^{\vth'}=\GR{A}{3}\oplus\GR{A}{3}\oplus\GR{A}{1}$, \ 
$\vth$ is $\gN$-regular and $|\vth|=|\vth'|=4$. 
Here
\\[.6ex] 
\centerline{
$\overrightarrow{\dim}(\vth)=(33,35,30,35)$ \ and \ $\overrightarrow{\dim}(\vth')=(33,32,36,32)$.
}
\\[.6ex]
Therefore $(\vth,\vth')$ is a friendly pair and $\vth'$ also admits a {\sf g.g.s.} 

{\bf 2}$^o$. For $\g$ of type $\GR{E}{6}$, we consider the following inner automorphisms of order $4$:

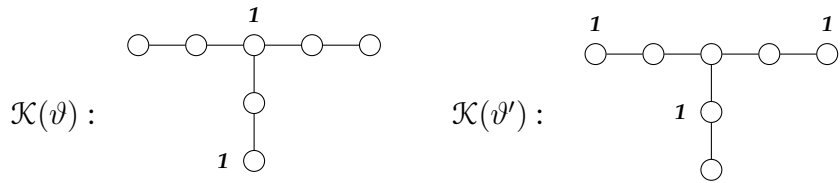
\centerline{
$\eus K(\vth)$ :\quad \raisebox{-3.8ex}{\begin{tikzpicture}[scale= .70, transform shape]
\tikzstyle{every node}=[circle, draw, fill=white!55] 
\node (a) at (0,0) {};
\node (b) at (1.1,0) {};
\node (c) [label=above:{\it\bfseries 1}] at (2.2,0) {};
\node (d) at (3.3,0) {};
\node (e) at (4.4,0) {};
\node (f)  at (2.2,-1.1) {};
\node (g) [label=left:{\it\bfseries 1}] at (2.2,-2.2) {};
\tikzstyle{every node}=[circle, draw, fill=violet!55] 
\foreach \from/\to in {a/b, b/c, c/d, d/e, c/f, f/g}  \draw[-] (\from) -- (\to);
\end{tikzpicture}}  \qquad
$\eus K(\vth')$ :\quad \raisebox{-3.8ex}{\begin{tikzpicture}[scale= .70, transform shape]
\tikzstyle{every node}=[circle, draw, fill=white!55] 
\node (a) [label=above:{\it\bfseries 1}] at (0,0) {};
\node (b) at (1.1,0) {};
\node (c) at (2.2,0) {};
\node (d)  at (3.3,0) {};
\node (e) [label=above:{\it\bfseries 1}] at (4.4,0) {};
\node (f)  [label=left:{\it\bfseries 1}] at (2.2,-1.1) {};
\node (g) at (2.2,-2.2) {};
\tikzstyle{every node}=[circle, draw, fill=violet!55] 
\foreach \from/\to in {a/b, b/c, c/d, d/e, c/f, f/g}  \draw[-] (\from) -- (\to);
\end{tikzpicture}}
}
\vskip1ex \noindent
Then $\g^\vth=\GR{A}{2}\oplus\GR{A}{2}\oplus\GR{A}{1}\oplus\te_1$ and 
$\g^{\vth'}=\GR{A}{3}\oplus\GR{A}{1}\oplus\te_2$.
Here $\vth$ is $\gN$-regular and
$\overrightarrow{\dim}(\vth)=\overrightarrow{\dim}(\vth')=(20,20,18,20)$.

{\bf 3}$^o$. For $\g=\mathfrak{sl}_{4n}$, $n\ge2$, we consider two {\bf outer} automorphisms of order $4$. The corresponding twisted affine Dynkin diagram is $\GR{A}{4n-1}^{(2)}$. It has $2n+1$ nodes.
\\[.6ex]  \indent
$\eus K(\vth)$ : \raisebox{-2.6ex}
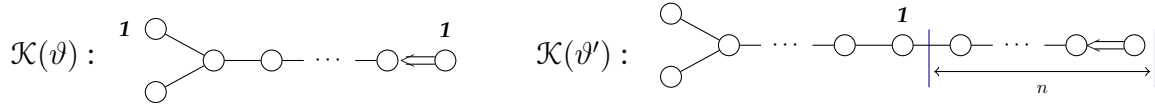
{\begin{tikzpicture}[scale= .70, transform shape]
\tikzstyle{every node}=[circle, draw, fill=white!55] 
\node (a) [label=left:{\it\bfseries 1}] at (0,.6) {};
\node (b) at (0,-.6) {};
\node (c) at (1.1,0) {};
\node (d) at (2.2,0) {};
\node (f) at (4.4,0) {};
\node (g) [label=above:{\it\bfseries 1}] at (5.5,0) {};
\tikzstyle{every node}=[circle] 
\node (e) at (3.3,0) {$\cdots$};
\node (h) at (4.72,0) {${\bf <}$};
\foreach \from/\to in {a/c, b/c, c/d, d/e, e/f}  \draw[-] (\from) -- (\to);
\draw (4.75, .07) -- +(.55,0);
\draw (4.75, -.07) -- +(.55,0);
\end{tikzpicture}} 
      \qquad
$\eus K(\vth')$ :\quad \raisebox{-2.8ex}{\begin{tikzpicture}[scale= .70, transform shape]
\tikzstyle{every node}=[circle, draw, fill=white!55] 
\node (a) at (0,.6) {};
\node (b) at (0,-.6) {};
\node (c) at (1.1,0) {};
\node (e) at (3.3,0) {};
\node (f) [label=above:{\it\bfseries 1}] at (4.4,0) {};
\node (g)  at (5.5,0) {};
\node (i)  at (7.7,0) {};
\node (j) at (8.8,0) {};
\tikzstyle{every node}=[circle] 
\node (d) at (2.2,0) {$\cdots$};
\node (h) at (6.6,0) {$\cdots$};
\node (k) at (8.02,0) {${\bf <}$};
\foreach \from/\to in {a/c, b/c, c/d, d/e, e/f, f/g, g/h, h/i}  \draw[-] (\from) -- (\to);
\draw (8.05, .07) -- +(.55,0);
\draw (8.05, -.07) -- +(.55,0);
{\color{darkblue}\draw (4.9, 0.3) -- (4.9,-0.8);}
{\color{darkblue}\draw (9.2, 0.3) -- (9.2,-0.8);}
\draw[<->] (5,-.5) -- (9.1,-.5) node[pos=.5,below] {\footnotesize $n$};
\end{tikzpicture}} 
\\[1ex]
Then $\g^{\vth}=\mathfrak{gl}_{2n}$ and $\g^{\vth'}=\spn\oplus\sone$. Here $\vth$ is $\gN$-regular,
and $\overrightarrow{\dim}(\vth)=\overrightarrow{\dim}(\vth')=(4n^2,4n^2,4n^2-1,4n^2)$. A similar example 
can be given for $\mathfrak{sl}_{4n-2}$.

{\bf 4}$^o$. A general idea is that if $\gcd(i,|\vth|)=1$, then $|\vth|=|\vth^i|$ and $\g^\vth=\g^{\vth^i}$.  Then it is not hard to provide examples, where $\vth$ and $\vth^i$ are not $G$-conjugate. For $|\vth|=5$, 
the dimension vector is of the form $\overrightarrow{\dim}(\vth)=(a,b,c,c,b)$ and hence 
$\overrightarrow{\dim}(\vth^2)=(a,c,b,b,c)$. Therefore, if $b\ne c$, then $\vth$ and $\vth^2$ are not 
$G$-conjugate, while $\dim\g^\vth=\dim\g^{\vth^2}=a$. For instance, this applies if $\g$ is of type $\GR{E}{6}$ 
and $\vth$ is $\gN$-regular, where $\overrightarrow{\dim}(\vth)=(16,16,15,15,16)$. 
\end{ex}

\end{document}